\documentclass[11pt,a4paper,reqno]{amsart}
\usepackage{amsfonts}\usepackage{amsmath,amssymb,amsthm,amsxtra}
\usepackage{float}
\usepackage{amssymb}
\usepackage{mathabx}
\usepackage{bbm}
\usepackage[colorlinks, linkcolor=blue,anchorcolor=Periwinkle,
citecolor=Red,urlcolor=Emerald]{hyperref}
\usepackage[usenames,dvipsnames]{xcolor}
\usepackage{enumitem}
\usepackage{geometry,array} 
\usepackage{graphicx}
\usepackage{subfigure}
\usepackage{bookmark}
\usepackage{tikz}\usetikzlibrary{matrix}
\usepackage{url}
\usepackage{dsfont}
\usepackage[colorinlistoftodos]{todonotes}
\usepackage{tablists} \restorelistitem

\usetikzlibrary{decorations.markings}
\tikzset{->-/.style={decoration={  markings,  mark=at position #1 with
    {\arrow{>}}},postaction={decorate}}}
\tikzset{-<-/.style={decoration={  markings,  mark=at position #1 with
    {\arrow{<}}},postaction={decorate}}}
\usepackage{extarrows}
\usepackage[all]{xy}
\usepackage{setspace}
\usepackage{thmtools}
\usepackage{thm-restate}
\usepackage{hyperref}
\usepackage{cleveref}

\newcommand{\id}{\operatorname{id}}

\newcommand{\mfd}{\mathbf{d}}

\newcommand{\mfx}{\mathbf{x}}

\newcommand{\mcA}{\mathcal{A}}

\newcommand{\mcF}{\mathcal{F}}

\newcommand{\mcM}{\mathcal{M}}

\newcommand{\mcT}{\mathcal{T}}

\newcommand{\mbN}{\mathbb{N}}

\newcommand{\mbQ}{\mathbb{Q}}

\newcommand{\mbT}{\mathbb{T}}

\newcommand{\mbZ}{\mathbb{Z}}

\newcommand{\IMR}{\operatorname{IMR}}
\theoremstyle{plain}
\newtheorem{theorem}{Theorem}[section]

\newtheorem{lemma}[theorem]{Lemma}
\newtheorem{corollary}[theorem]{Corollary}
\newtheorem{proposition}[theorem]{Proposition}

\theoremstyle{definition}
\newtheorem{definition}[theorem]{Definition}

\newtheorem{example}[theorem]{Example}

\newtheorem{remark}[theorem]{Remark}
\newtheorem{question}[theorem]{Question}
\newtheorem{notations}[theorem]{Notations}
\numberwithin{equation}{section}
\newtheorem{definition-proposition}[theorem]{Definition-Proposition}

\begin{document}

\title {Mutation invariants of cluster algebras of rank 2}

\date{\today}

\author{Zhichao Chen}
\address{School of Mathematical Sciences\\ University of Science and Technology of China \\ Hefei, Anhui 230026, P. R. China}
\email{czc98@mail.ustc.edu.cn}
\author{Zixu Li}
\address{Department of Mathematical Sciences\\ Tsinghua University\\ Beijing 100080, P. R. China}
\email{lizx19@mails.tsinghua.edu.cn}
\maketitle

\begin{abstract}
 We consider the mutation invariants of cluster algebras of rank 2. We characterize the mutation invariants of finite type. Two examples are provided for the affine type and we prove the non-existence of Laurent mutation invariants of non-affine type. As an application, a class of Diophantine equations encoded with cluster algebras are studied.
\end{abstract}
\tableofcontents
\section*{Introduction}
\label{sec1}
Cluster algebras were first introduced  by Fomin and Zelevinsky in \cite{FZ1,FZ2} to investigate total positivity of Lie groups and canonical bases of quantum groups. They are subalgebras of rational function fields over a certain field. Nowadays, cluster algebras are closely related to different subjects in mathematics, such as higher Teichm{\"u}ller theory \cite{FG1}, representation theory \cite{BIRS,HL,KY}, Poisson geometry \cite{GSV}, integrable system \cite{KNS}, number theory \cite{P,BBH, PZ,L,LLRS,H,GM,BL} and so on.

The relations between cluster algebras and Diophantine equations are investigated. Propp \cite{P}, Beineke-Br{\"u}stle-Hille \cite{BBH}, Peng-Zhang \cite{PZ}, Lee-Li-Rabideau-Schiffler \cite{LLRS} and Huang \cite{H} studied the relations between the famous Markov equation
\begin{align}
	X^2+Y^2+Z^2=3XYZ\label{markov}
\end{align} and the once-punctured torus cluster algebras. Afterwards,
Lampe \cite{L} exhibited how cluster mutations generate all solutions
to a variant of Markov Diophantine equation 
\begin{align}
X^2+Y^4+Z^4+2XY^2+2XZ^2=7XY^2Z^2.\label{lampe-equation}
\end{align} Following his work, Bao-Li \cite{BL} gave a criterion to determine  the solutions in the orbit of the initial solution under the actions of a certain group. In addition, Gyoda-Matsushita \cite{GM} solved the generalized Markov equations
\begin{align}
	X^2+Y^2+Z^2+k_1XY+k_2YZ+k_3XZ=(3+k_1+k_2+k_3)XYZ\label{GM-equation}
\end{align} and studied the structure of generalized cluster algebras behind them. Motivated by them, we wish to find a class of Diophantine equations encoded with cluster algebras instead of a single one. There is an important notion called \emph{(Laurent) mutation invariant}, which appears vaguely in \cite{L} and we define it formally in a rational function field, see \Cref{mutation invariant}. 

In this paper, we study and classify mutation invariants of cluster algebras of rank 2. Let $\mcA$ be a cluster algbera of rank 2 with the initial exchange matrix 
$$\begin{pmatrix}
	0 & m \\ -n & 0
\end{pmatrix},$$ where $m,n \in \mbN$. We define M-actions (\Cref{M}) and find an important link with cluster mutations as follows. Note that the M-actions coincide with the \emph{initial seed mutations} without coefficients, which first appeared in \cite{NZ,RS} and were defined formally in \cite{FG2}.
\begin{theorem}[\Cref{cluster mutations are equivalent to M-actions rank 2 case}]
	For any $k \in \mbN$, the following two identities hold. 
	\begin{enumerate}
		\item $\mu_i(\mu_j\mu_i)^k(x_1,x_2)=\widetilde{\mcM_i}(\widetilde{\mcM_j}\widetilde{\mcM_i})^k(x_1,x_2)$, where $(i,j)=(1,2), (2,1)$.
		\item $(\mu_i\mu_j)^k(x_1,x_2)=(\widetilde{\mcM_j}\widetilde{\mcM_i})^k(x_1,x_2)$, where $(i,j)=(1,2), (2,1)$.
	\end{enumerate}
\end{theorem}
Recall that $\mcA$ is of finite type if $mn\leq 3$, affine type if $mn=4$, and non-affine type if $mn\geq 5$, cf. \cite{FZ2, N}. Let $(x_1,x_2)$  be the initial cluster. Based on the Laurent phenomenon \cite[Theorem 3.1]{FZ1} of cluster variables and M-actions, a complete characterization of mutation invariants of finite type is given as follows.
\begin{theorem}[\Cref{M action provides mutation invariants rank 2}]\

Let $\mcA$ be of finite type with $m$ clusters
	$(c_{1;i}(x_1,x_2),c_{2;i}(x_1,x_2))^m_{i=1}.$
	Then a non-constant rational function $\mcT(x_1,x_2)$ is a mutation invariant of $\mcA$ if and only if there exist a symmetric polynomial $\Phi(X_1,\cdots,X_m)$ of $m$ variables over $\mbQ$ and a rational function $F(X_1,X_2)$, such that 
\begin{align}
		\mcT(x_1,x_2)=&{}\Phi(F(c_{1;1}(x_1,x_2),c_{2;1}(x_1,x_2)),\cdots,F(c_{1;m}(x_1,x_2),c_{2;m}(x_1,x_2))).\notag
	\end{align} 
\end{theorem} 
 A positive answer to the existence of (Laurent) mutation invariants of affine type is proved by giving two examples, see \Cref{example1} and \Cref{example2}.
However, different from the finite type and the affine type, we prove the non-existence of Laurent mutation invariants of non-affine type by using $\mfd$-vectors.
\begin{theorem}[\Cref{Thm4.3}]
	There does not exist a Laurent mutation invariant of non-affine type.
\end{theorem}

The paper is organized as follows. In Section 1, we review basic definitions and properties about cluster algebras, and introduce the definitions of  mutation invariants (\Cref{mutation invariant}) and cluster algebras with IMR condition (\Cref{IMR condition}). In Section 2, we define M-actions (\Cref{M}), which are closely connected with cluster mutations, see \Cref{cluster mutations are equivalent to M-actions rank 2 case}. With the help of M-actions, we will prove \Cref{M action provides mutation invariants rank 2}. Two examples of mutation invariants for affine type are given, see \Cref{example1} and \Cref{example2}. As the end of Section 2, we prove \Cref{Thm4.3} by using the general term formulas of $\mfd$-vectors (\Cref{d}). Some Diophantine equations arising from mutation invariants are solved in Section 3.

\section*{Conventions}
In this paper, we denote $\mbZ$, $\mbN$ and $\mbQ$ the integer ring, natural number set and the rational number field respectively. Let $\mcF$ be a field that is isomorphic to the rational function field of $n$ variables over $\mbQ$. 
Denote $\mbQ[x_1,\dots,x_n]$ and $\mbQ(x_1,\dots,x_n) $ the polynomial ring and the rational function field of $x_1,\dots,x_n$ respectively. Let $\mbQ^*(x_1,\cdots,x_n)=\mbQ(x_1,\cdots,x_n)-\{0\}$. The set of Laurent polynomials of $x_{1},\dots,x_{n}$ over $\mbQ$ is $\mbQ[x_{1}^{\pm 1},\dots,x_{n}^{\pm 1}]$. For an integer $a\in \mbZ$, let $[a]_{+}=\max(a,0)$. An $n\times n$  integer matrix $B$ is called \emph{skew-symmetrizable} if there is a positive integer diagonal matrix $D$ such that $DB$ is skew-symmetric. 
\section{Preliminaries}
\label{sec2}
Firstly, we provide a brief overview of mutations and cluster algebras, along with their basic definitions and properties as described in \cite{FZ4, N}.

 A pair $\Sigma=(\mfx,B)$ is called a \emph{seed} if $\mfx=(x_1,\dots,x_n)$ is an $n$-tuple of algebraically independent and generating elements in $\mcF$ and $B=(b_{ij})_{n\times n}$ is a skew-symmetrizable integer matrix. The set $\mfx$ is called a \emph{cluster}, each $x_{i}$ is called \emph{cluster variable} and $B$ is called the \emph{exchange matrix}. One can see that $\mcF=\mbQ(x_1,\dots,x_n)$. Fix $k\in \{1,\dots,n\}$, a new seed $\mu_{k}(\Sigma)=(\mfx^{\prime},B^{\prime})$ is defined as follows, where $\mfx^{\prime}=(x_{1}^{\prime},\dots,x_{n}^{\prime})$ is given by 
	\begin{align} \notag 
		x_{i}^{\prime}=\left\{
		\begin{array}{ll}
			x_{k}^{-1}(\prod\limits_{j=1}^{n}x_{j}^{[b_{jk}]_{+}}+\prod\limits_{j=1}^{n}x_{j}^{[-b_{jk}]_{+}}), &   i=k, \\
			x_{i}, &   i \neq k. 
		\end{array} \right.
	\end{align}
and the entries of $B^{\prime}=(b_{ij}^{\prime})_{n\times n}$ are given by 
	\begin{align} \notag 
		b_{ij}^{\prime}=\left\{
		\begin{array}{ll}
			-b_{ij}, &   i=k \;\;\mbox{or}\;\; j=k, \\
			b_{ij}+[b_{ik}]_{+}b_{kj}+b_{ik}[-b_{kj}]_{+}, &   \mbox{otherwise}. 
		\end{array} \right.
	\end{align}
For the well-definedness of this new seed, we refer to \cite{N}. It is not hard to see that $\mu_{k}$ is an involution, and we call $\mu_k(\Sigma)$ the \emph{k-direction mutation} of $\Sigma$. 

Let $\mbT_{n}$ be an $n$-regular tree whose edges are labeled by $1,\dots,n$. Let $t$ and $t^{\prime}$ be vertices of $\in \mbT_{n}$ which are connected by a $k$-labeled edge, and we denote it by $t \stackrel{k}{\longleftrightarrow } t^{\prime}$. A set of seeds $\mathbf{\Sigma}=\{\Sigma_{t}=(\mfx_{t},B_{t})| t\in \mbT_{n}\}$ indexed by vertices of $\mbT_{n}$ is called a \emph{cluster pattern} if $\Sigma_{t^{\prime}}=\mu_{k}(\Sigma_{t})$ holds for any $t \stackrel{k}{\longleftrightarrow} t^{\prime}$. In this case, the seed  $\Sigma_{t}=(\mfx_{t},B_{t}) $ is denoted by $\mathbf{x_{t}}=(x_{1;t},\dots,x_{n;t})$, $B_{t}=(b_{ij;t})_{n\times n}$. A cluster pattern $\mathbf{\Sigma}$ is called of \emph{finite type} (resp. \emph{infinite type}) if it contains finitely (resp. infinitely) many distinct seeds. 

The \emph{cluster algebra} $\mcA=\mcA(\mathbf{\Sigma})$ is the $\mbQ$-subalgebra of $\mcF$ generated by all cluster variables $\{x_{i;t}|i=1,\dots,n;t\in\mbT_{n}\}$. Here $n$ is called the \emph{rank} of $\mcA$ or $\mathbf{\Sigma}$. The cluster algebra $\mcA(\mathbf{\Sigma})$ is called of \emph{finite type} (resp. \emph{infinite type}) if $\mathbf{\Sigma}$ is of finite type (resp. infinite type).
Fix an arbitrary vertex $t_0\in \mbT_{n}$, denote $\mfx_{t_0}=(x_1,\dots,x_n)$, and we call $\Sigma_{t_0}$ the \emph{initial seed}. Then the set of all clusters labeled by $t$ can be written as  
$$\mfx_t=((c_{1;t}(x_1,\cdots,x_n),\cdots,c_{n;t}(x_1,\cdots,x_n)),$$ where $c_{i,t}(x_1,\cdots,x_n)$ is a Laurent polynomial in $x_1,\cdots,x_n$ with integer coefficients, see \cite[Theorem 3.1]{FZ1}. 

For any cluster variable $x_{i;t}$, the integer vector $\mfd_{i;t}=(d_{ji;t})_{j=1}^n$ is called the \emph{denominator vector} ($\mfd$-vector) of $x_{i;t}$, where $-d_{ji;t}$ is the lowest degree of $x_{j}$ in the Laurent polynomial expression of $x_{i;t}$ in $\mfx_{t_0}$. That is,  
\begin{align}
	x_{i;t}=\dfrac{N_{i;t}(x_1,\dots, x_n)}{x_1^{d_{1i;t}}\cdots x_n^{d_{ni;t}}},\label{express of d-vector}
\end{align} where $N_{i;t}(x_1,\dots, x_n)$ is a polynomial with coefficients in $\mathbb{Z}$ which is not divisible by any $x_j$. The recurence relations for $\mfd$-vectors are given as follows, cf. \cite[Section 4.3]{FZ3}.
		\begin{align} \label{d-vector}
		\mfd_{l;t^{\prime}}=\left\{
		\begin{array}{ll}
			\mfd_{l;t}, &   l\neq k, \\
			-\mfd_{k;t}+\text{max}\ (\sum\limits_{i=1}^n[b_{ik;t}]_{+}\mfd_{i;t},\sum\limits_{i=1}^n[-b_{ik;t}]_{+}\mfd_{i;t}), &   l=k,
		\end{array} \right.
	\end{align}
	for $t \stackrel{k}{\longleftrightarrow} t^{\prime}$.

\begin{definition}\label{mutation invariant}
	 A non-constant and reduced rational function $\mcT(x_1,\dots,x_n)\in \mbQ(x_1,\dots,x_n)$ is called a \emph{mutation invariant of $\mcA$} if for any $t \in \mbT_{n}$,
	\begin{equation}
		\mcT(x_{1},\cdots,x_{n})=\mcT(x_{1;t},\cdots,x_{n;t}).\notag
	\end{equation}
	Furthermore, if $\mcT(x_1,\dots,x_n) \in \mbQ[x_{1}^{\pm 1},\dots,x_{n}^{\pm 1}]$, it is called a \emph{Laurent mutation invariant of $\mcA$}. 
\end{definition}
\begin{remark}
	Without ambiguity, we identify the initial cluster variables $x_1,\dots,x_n$ with the variables of rational function field $\mbQ(x_1,\dots,x_n)$. In addition, the notion of mutation invariants first appeared vaguely in \cite[Proposition 2.2]{L}.
\end{remark}
\begin{definition}\label{*}  Let
	$(f_1,\cdots,f_n) \in  \mbQ^*(x_1,\cdots,x_n)^{\times n}$ and $\mathbf{\Sigma}$ be a cluster pattern. For any $t \in \mbT_{n}$ and $k \in \{1,\cdots,n\}$, define a map $\mu_k^t$ from $\mbQ^*(x_1,\cdots,x_n)^{\times n}$ to itself in the following way
	\begin{align} \notag 
		\mu_k^t(f_i)=\left\{
		\begin{array}{ll}
			f_{k}^{-1}(\prod\limits_{j=1}^{n}f_{j}^{[b_{jk;t}]_{+}}+\prod\limits_{j=1}^{n}f_{j}^{[-b_{jk;t}]_{+}}), &   i=k, \\
			f_i, &   i \neq k, 
		\end{array} \right.
	\end{align}
	where $b_{jk;t}$ is the $(j,k)$-component of $B_t$.
\end{definition}
Notice that $\mu_k^t$ is an involution and if $((f_1,\cdots,f_n),B_t)$ is a seed of the cluster algebra $\mcA$, then $\mu_k^t$ can be viewed as the cluster mutation along the $k$-th direction of $((f_1,\cdots,f_n),B_t)$.
\begin{definition}\label{IMR condition}
	A \emph{cluster algebra with $\IMR$ condition} is a cluster algebra $\mcA$ which satisfies
	$\mu_k^t=\mu_k^{t_{0}},$ for any  $t \in \mbT_{n}$ and $k \in \{1,\cdots,n\}.$
\end{definition}
Here IMR means \emph{invariant mutation rules}. Given a cluster algebra with $\IMR$ condition, denote $\mu^t_k$ by $\mu_k$ for any  $t \in \mbT_{n}$ and $k \in \{1,\cdots,n\}.$ It can be checked directly that the following lemma holds.
\begin{lemma}\label{equivalent} Let $\mcA$ be a cluster algebra of rank $n$ with $\IMR$ condition and $\mcT(x_{1},\cdots,x_{n})$ be a non-constant rational function. The following are equivalent.
\begin{enumerate}
	\item $\mcT(x_{1},\cdots,x_{n})$ is a mutation invariant of $\mcA$.
	\item $\mcT(x_{1},\cdots,x_{n})=\mcT(\mu_k(x_{1},\cdots,x_{n}))$, for any $k\in \{1,\dots,n\}$.
\end{enumerate}
\end{lemma}
As the simplest examples of cluster algebras with IMR condition, the next lemma can be obtained immediately.
\begin{lemma}
	All the cluster algebras of rank 2 are cluster algebras with $\IMR$ condition. 
\end{lemma}
\begin{proof}
	The mutation equivalence class of the initial exchange matrix $B_{t_0}$ with rank 2 is $\{B_{t_0},-B_{t_0}\}$.
\end{proof}
\begin{remark}
	For cluster algebras of rank $2$, by \Cref{*} and the cluster mutation rules, it is direct that each cluster  can be expressed by the initial cluster $(x_1,x_2)$ either $\mu_i(\mu_j\mu_i)^k(x_1,x_2)$ or $(\mu_j\mu_i)^k(x_1,x_2)$, where $k\in \mbN$, $(i,j)=(1,2)\ \text{or}\ (2,1)$.
\end{remark}
\begin{example} There are only two cluster algebras of rank 3 with $\IMR$ condition \cite{L}, whose initial exchange matrix (up to permutation) is respectively
		\begin{center}
			$B_{t_0}=\begin{pmatrix}0 & 2 & -2\\ -2 & 0 & 2\\ 2 & -2 & 0\end{pmatrix}$ and $\begin{pmatrix}0 & 1 & -1\\ -4 & 0 & 2\\ 4 & -2 & 0\end{pmatrix}$.
		\end{center}
By direct calculation, the mutation equivalence class is both $\{B_{t_0},-B_{t_0}\}$.
\end{example}
\section{Mutation invariants of cluster algebras of rank $2$}
\label{sec3}
In this section, we define M-actions and find the relations between them and cluster mutations. Moreover, we use them to study and classify the  mutation invariants of cluster algebras of rank $2$.
\subsection{M-actions and mutation invariants }\

Let $\mcA$ be a cluster algebra with the initial exchange matrix 
\begin{align}
\begin{pmatrix}
	0 & m \\ -n & 0
\end{pmatrix},\ \text{where}\ m,n \in \mbN.\label{matrix}
\end{align}
Then the mutation rules on $\mbQ^*(x_1,x_2)\times \mbQ^*(x_1,x_2)$ as \Cref{*} are
$$\mu_1(f,g)=(m_1(f,g),g),\ \mu_2(f,g)=(f,m_2(f,g)),$$
where $m_1(f,g)=\dfrac{g^n+1}{f}$, $m_2(f,g)=\dfrac{f^m+1}{g}$ and $f,g \in \mbQ^*(x_1,x_2)$.
\begin{definition}\label{M}
Let four maps be as follows, where $\mcM_i$ is the map from $\mbQ(x_1,x_2)$ to itself and $\widetilde{\mcM_i}$ is the map from $\mbQ(x_1,x_2) \times \mbQ(x_1,x_2)$ to itself.
\begin{enumerate}
	\item $\mcM_1:
	f(x_1,x_2) \mapsto f(m_1(x_1,x_2),x_2)$,
	\item $\mcM_2:
	f(x_1,x_2) \mapsto f(x_1, m_2(x_1,x_2))$,
	\item $\widetilde{\mcM_1}:
	(f(x_1,x_2),g(x_1,x_2)) \mapsto (f(m_1(x_1,x_2),x_2),g(m_1(x_1,x_2),x_2)$,
	\item $\widetilde{\mcM_2}:
	(f(x_1,x_2),g(x_1,x_2)) \mapsto (f(x_1,m_2(x_1,x_2)),g(x_1,m_2(x_1,x_2))$.
\end{enumerate}
Here $\widetilde{\mcM_1}$ and $\widetilde{\mcM_2}$ are called \emph{M-actions}.
\end{definition}
It can be checked directly that the maps in \Cref{M} are all involutions. The following lemma exhibits the relations between $\mcM_i$ and $\widetilde{\mcM_i}$, which can be proved by definition.
\begin{lemma}\label{M-actions}
	Let $f_1(x_1,x_2), f_2(x_1,x_2) \in \mathbb{Q}(x_1,x_2)$. For any $k\in \mbN$,
	\begin{enumerate}
		\item $f_1((\widetilde{\mcM_i}\widetilde{\mcM_j})^k(x_1,x_2))=(\mcM_i\mcM_j)^k(f_1(x_1,x_2))$.
		\item $f_1(\widetilde{\mcM_j}(\widetilde{\mcM_i}\widetilde{\mcM_j})^k(x_1,x_2))=\mcM_j(\mcM_i\mcM_j)^k(f_1(x_1,x_2))$.
		\item $(\widetilde{\mcM_i}\widetilde{\mcM_j})^k(f_1(x_1,x_2), f_2(x_1,x_2))=((\mcM_i\mcM_j)^k(f_1(x_1,x_2)),(\mcM_i\mcM_j)^k(f_2(x_1,x_2)))$, where $(i,j)=(1,2), (2,1)$.
		\item $\widetilde{\mcM_j}(\widetilde{\mcM_i}\widetilde{\mcM_j})^k(f_1(x_1,x_2), f_2(x_1,x_2))=(\mcM_j(\mcM_i\mcM_j)^k(f_1(x_1,x_2)),\mcM_j(\mcM_i\mcM_j)^k(f_2(x_1,x_2)))$, where $(i,j)=(1,2), (2,1)$.
	\end{enumerate}
\end{lemma}
\begin{notations}
	\label{pp}
	Let $p_0=(x_1,x_2)\in \mbQ(x_1,x_2) \times \mbQ(x_1,x_2)$. For technicality, we artificially assume that $m_1(p_{-2})=x_1$ and  $m_2(p_{-1})=x_2$. 
	For any $k\in \mbN$, denote $$p_{2k}=(m_1(p_{2k-2}),m_2(p_{2k-1})),\ p_{2k+1}=(m_1(p_{2k}),m_2(p_{2k-1})).$$ Under the notations, we get a lemma which reveals the relation between $p_{i}$ and $\mu_{j}$.
\end{notations}
\begin{lemma}\label{little p} The following two identities hold.
	\begin{enumerate}
		\item $p_{2k}=(\mu_2\mu_1)^k(x_1,x_2)$, where $k\in \mbN$.
		\item $p_{2k+1}=\mu_1(\mu_2\mu_1)^k(x_1,x_2)$, where $k\in \mbN$.
	\end{enumerate}
\end{lemma}
\begin{proof}
	We take the double induction on $k$. It is clear that 
	$$p_0=(x_1,x_2)=(\mu_2\mu_1)^0(x_1,x_2),\ p_1=(m_1,x_2)=\mu_1(\mu_2\mu_1)^0(x_1,x_2).$$ 
	Assume that the two identities hold for $k=n$. When $k=n+1$, we have 
	\begin{align}\notag
	(\mu_2\mu_1)^{n+1}(x_1,x_2)=&{}
\mu_2(\mu_1(\mu_2\mu_1)^n(x_1,x_2)) \notag \\
		=&{}\mu_2(m_1(p_{2n}),m_2(p_{2n-1})) \notag \\
		=&{}(m_1(p_{2n}),m_2(m_1(p_{2n}),m_2(p_{2n-1}))) \notag \\
		=&{}(m_1(p_{2n}),m_2(p_{2n+1})) \notag \\
		=&{}p_{2n+2} \notag,
	\end{align}
	and
	\begin{align}\notag
		\mu_1(\mu_2\mu_1)^{n+1}(x_1,x_2)
		 \notag 
		=&{} \mu_1(m_1(p_{2n}),m_2(p_{2n+1})) \notag \\
		=&{}(m_1(m_1(p_{2n}),m_2(p_{2n+1})),m_2(p_{2n+1})) \notag \\
		=&{}(m_1(p_{2n+2}),m_2(p_{2n+1})) \notag \\
		=&{}p_{2n+3}. \notag
	\end{align}
	Hence, the lemma holds.
\end{proof}
\begin{notations}
	\label{PP}
Let $P_0=p_0=(x_1,x_2)\in \mbQ(x_1,x_2) \times \mbQ(x_1,x_2)$. For technicality, we artificially assume that $m_1(P_{-1})=x_1$ and  $m_2(P_{-2})=x_2$.
	For any $k\in \mbN$, denote $$P_{2k}=(m_1(P_{2k-1}),m_2(P_{2k-2})),\ P_{2k+1}=(m_1(P_{2k-1}),m_2(P_{2k})).$$ Similarly, we get a lemma which reveals the relations between $P_{i}$ and $\mu_{j}$.
\end{notations}

\begin{lemma}\label{big P}The following two identities hold.
	\begin{enumerate}
		\item $P_{2k}=(\mu_1\mu_2)^k(x_1,x_2)$, where $k\in \mbN$.
		\item $P_{2k+1}=\mu_2(\mu_1\mu_2)^k(x_1,x_2)$, where $k\in \mbN$.
	\end{enumerate}
\end{lemma}
\begin{proof}
	It is similar to the proof of \Cref{little p}. 
\end{proof}
Then the relations among $\mcM_i$, $p_i$ and $P_i$ are as follows.
\begin{lemma}\label{m p}
	The following identities hold.
	\begin{enumerate}
		\item 
		$\mcM_2(m_1(p_{2k}))=m_1(P_{2k+1})$, $\mcM_2(m_2(p_{2k+1}))=m_2(P_{2k+2})$, where $k\in \mbN$.
		\item 	
		$\mcM_1(m_1(P_{2k+1}))=m_1(p_{2k+2})$, $\mcM_1(m_2(P_{2k}))=m_2(p_{2k+1})$, where $k\in \mbN$.
	\end{enumerate}
\end{lemma}
\begin{proof}
	Without loss of generality, we only focus on the first two identites since the second ones are similar. We take the double induction on $k$. 
	When $k=0$, we have
\[
\begin{array}{l}
\mcM_2(m_1(p_{0}))=m_1(x_1,m_2)=m_1(P_{1}), \\
\mcM_2(m_2(p_1))=m_2(m_1(x_1,m_2),m_2)=m_2(P_2).	
\end{array}
\]
Assume that the two identities hold for $k=n$, that is 
	$$\mcM_2(m_1(p_{2n}))=m_1(P_{2n+1}),\ \mcM_2(m_2(p_{2n+1}))=m_2(P_{2n+2}).$$ 
	When $k=n+1$, by \Cref{pp} and \Cref{PP}, we have 
	\begin{align}
		\mcM_2(m_1(p_{2n+2}))=
		&{}\mcM_2(m_1(m_1(p_{2n}),m_2(p_{2n+1}))) \notag \\
		=&{} m_1(\mcM_2(m_1(p_{2n})),\mcM_2(m_2(p_{2n+1}))) \notag \\
		=&{} m_1(m_1(P_{2n+1}),m_2(P_{2n+2})) \notag \\
		=&{}m_1(P_{2n+3}), \notag
	\\\mcM_2(m_2(p_{2n+3}))=
		&{}\mcM_2(m_2(m_1(p_{2n+2}),m_2(p_{2n+1}))) \notag \\=&{} m_2(\mcM_2(m_1(p_{2n+2})),\mcM_2(m_2(p_{2n+1}))) \notag \\
		=&{}m_2(m_1(P_{2n+3}),m_2(P_{2n+2}))\notag \\
		=&{}m_2(P_{2n+4}). \notag
	\end{align} 
	Hence, the lemma holds for any $k\in \mbN$.
\end{proof}
Now, there are important relations between M-actions $\widetilde{\mcM_i}$ and mutations $\mu_j$.
\begin{proposition}\label{cluster mutations are equivalent to M-actions rank 2 case}
	For any $k \in \mbN$, the following two identities hold. 
	\begin{enumerate}
		\item $\mu_i(\mu_j\mu_i)^k(x_1,x_2)=\widetilde{\mcM_i}(\widetilde{\mcM_j}\widetilde{\mcM_i})^k(x_1,x_2)$, where $(i,j)=(1,2), (2,1)$.
		\item $(\mu_i\mu_j)^k(x_1,x_2)=(\widetilde{\mcM_j}\widetilde{\mcM_i})^k(x_1,x_2)$, where $(i,j)=(1,2), (2,1)$.
	\end{enumerate}
\end{proposition}
\begin{proof}
	Without loss of generality, we only focus on the two identities when $(i,j)=(1,2)$ since the other case is similar. We take the double induction on $k$.
	When $k=0$, it is clear that
\[
\begin{array}{l}
(\widetilde{\mcM_2}\widetilde{\mcM_1})^0(x_1,x_2)=(x_1,x_2)=(\mu_1\mu_2)^0(x_1,x_2),\\
\widetilde{\mcM_1}(\widetilde{\mcM_2}\widetilde{\mcM_1})^0(x_1,x_2)=(m_1,x_2)=\mu_1(\mu_1\mu_2)^0(x_1,x_2).	
\end{array}
\] 
Assume that the two identities hold for $k=n$. According to \Cref{little p} and \Cref{big P}, we have
\[
\begin{array}{l}
(\widetilde{\mcM_2}\widetilde{\mcM_1})^n(x_1,x_2)=(m_1(P_{2n-1}),m_2(P_{2n-2})),\\
\widetilde{\mcM_1}(\widetilde{\mcM_2}\widetilde{\mcM_1})^n(x_1,x_2)=(m_1(p_{2n}),m_2(p_{2n-1})).	
\end{array}
\]	 
When $k=n+1$, by \Cref{M-actions} and \Cref{m p}, we get  
	\begin{align}
		(\widetilde{\mcM_2}\widetilde{\mcM_1})^{n+1}(x_1,x_2)=
		&{}\widetilde{\mcM_2}(\widetilde{\mcM_1}(\widetilde{\mcM_2}\widetilde{\mcM_1})^n(x_1,x_2)) \notag \\
		=&{}\widetilde{\mcM_2}(m_1(p_{2n}),m_2(p_{2n-1})) \notag \\
		=&{}(m_1(P_{2n+1}),m_2(P_{2n})) \notag \\
		=&{}(\mu_1\mu_2)^{n+1}(x_1,x_2), \notag
	\end{align}
	and
	\begin{align}
		\widetilde{\mcM_1}(\widetilde{\mcM_2}\widetilde{\mcM_1})^{n+1}(x_1,x_2)=
		&{}\widetilde{\mcM_1}(m_1(P_{2n+1}),m_2(P_{2n})) \notag \\
		=&{}(m_1(p_{2n+2}),m_2(p_{2n+1})) \notag \\
		=&{}\mu_1(\mu_2\mu_1)^{n+1}(x_1,x_2). \notag
	\end{align}
	Hence, the proposition holds for any $k\in \mbN$.
\end{proof}
Since $\widetilde{\mcM_i}^2=\id$, a direct corollary of \Cref{cluster mutations are equivalent to M-actions rank 2 case} is as follows.
\begin{corollary}\label{M action is stable on seed set rank 2 case}
	The M-actions $\widetilde{\mcM_1}$ and $\widetilde{\mcM_2}$ are bijections over the set of all clusters of the cluster algebra $\mathcal{A}$ of rank 2.
\end{corollary}
Moreover, according to $\mcM_i^2=\id$ and \Cref{M action is stable on seed set rank 2 case}, we have the following corollary.
\begin{corollary}\label{cor2}
	The maps $\mcM_1$ and $\mcM_2$ are bijections over the set of all cluster variables of the cluster algebra $\mathcal{A}$ of rank 2.
\end{corollary}
\begin{remark}
	Note that \Cref{cluster mutations are equivalent to M-actions rank 2 case}, \Cref{M action is stable on seed set rank 2 case} and \Cref{cor2} are independent of whether the cluster algebra is of finite type.
\end{remark}
\begin{question}
	When the cluster algebra with IMR condition is of rank $n\ (n\geq 3)$, one can similarly define M-actions $\widetilde{\mcM_i},\ i=1,\dots,n$. What is the relation between them and mutations $\mu_j$? 
\end{question}
In the following, we aim to prove our main theorem.
\begin{theorem}
	\label{M action provides mutation invariants rank 2} 
	Let $\mcA$  be of finite type with $m$ clusters 
	$(c_{1;i}(x_1,x_2),c_{2;i}(x_1,x_2))^m_{i=1}.$
	Then a non-constant rational function $\mcT(x_1,x_2)$ is a mutation invariant of $\mcA$ if and only if there exist a symmetric polynomial $\Phi(X_1,\cdots,X_m)$ of $m$ variables over $\mbQ$ and a rational function $F(X_1,X_2)$, such that 
	\begin{align}\label{main formula}
		\mcT(x_1,x_2)=&{}\Phi(F(c_{1;1}(x_1,x_2),c_{2;1}(x_1,x_2)),\cdots,F(c_{1;m}(x_1,x_2),c_{2;m}(x_1,x_2))).
	\end{align} 
\end{theorem}
\begin{proof}
We begin with proving the sufficient condition by assuming that the equality \eqref{main formula} holds. Firstly, according to \Cref{equivalent} and \Cref{cluster mutations are equivalent to M-actions rank 2 case}, we only need to check that 
\begin{align}\mcT(x_1,x_2)=\mcT(\widetilde{\mcM_i}(x_1,x_2)),\ i=1,2. \notag
\end{align}  
Furthermore, by \Cref{M action is stable on seed set rank 2 case} and the property of symmetric polynomial $F$, we get the equalities as follows
\begin{equation}\label{M action for mutation invariants}
\begin{array}{l}
 \ \ \ \mcT(\widetilde{\mcM_i}(x_1,x_2))\\
=\Phi(F(c_{1;1}(\widetilde{\mcM_i}(x_1,x_2)),c_{2;1}(\widetilde{\mcM_i}(x_1,x_2))),\cdots,F(c_{1;m}(\widetilde{\mcM_i}(x_1,x_2)),c_{2;m}(\widetilde{\mcM_i}(x_1,x_2))))\\
=\Phi(F(\widetilde{\mcM_i}(c_{1;1}(x_1,x_2),c_{2;1}(x_1,x_2))),\cdots,F(\widetilde{\mcM_i}(c_{1;m}(x_1,x_2),c_{2;m}(x_1,x_2))))\\
=\Phi(F(c_{1;\sigma_i(1)}(x_1,x_2),c_{2;\sigma_i(1)}(x_1,x_2)),\cdots,F(c_{1; \sigma_i(m)}(x_1,x_2),c_{2;\sigma_i(m)}(x_1,x_2)))\\
=\Phi(F(c_{1;1}(x_1,x_2),c_{2;1}(x_1,x_2)),\cdots,F(c_{1;m}(x_1,x_2),c_{2;m}(x_1,x_2))) \\
=\mcT(x_1,x_2),
	\end{array} \notag
\end{equation}
where $i=1,2$ and $\sigma_i\in S_m$ is a permutation induced by $\widetilde{\mcM_i}$ on the index set $\{1,\cdots,m\}$ of all clusters of $\mcA$.

	Then, we prove the necessary condition and assume that $\mcT(X_1,X_2)$ is a mutation invariant of $\mcA$. Take the symmetric polynomial $$\Phi(X_1,\cdots,X_m)=\dfrac{X_1+\dots+X_m}{m}$$ and the rational function $F(X_1,X_2)=\mcT(X_1,X_2)$. Since $\mcT(X_1,X_2)$ is a mutation invariant, we get
	\begin{align}
		&{}\Phi(F(c_{1;1}(x_1,x_2),c_{2;1}(x_1,x_2)),\cdots,F(c_{1;m}(x_1,x_2),c_{2;m}(x_1,x_2)))\notag \\=&{} \Phi(\mcT(c_{1;1}(x_1,x_2),c_{2;1}(x_1,x_2)),\cdots,\mcT(c_{1;m}(x_1,x_2),c_{2;m}(x_1,x_2)))\notag \\=&{} \Phi(\mcT(x_1,x_2),\cdots,\mcT(x_1,x_2))) \notag \\=&{} \mcT(x_1,x_2).\notag
	\end{align}
\end{proof}
\begin{remark}
	In \Cref{M action provides mutation invariants rank 2}, if $F(X_1,X_2)$ is a polynomial, $\mcT(x_1,x_2)$ is a Laurent mutation invariant, that is $\mcT(x_1,x_2)\in \mathbb{Q}[x_1^{\pm1},x_2^{\pm1}].$
	
\end{remark}

\subsection{Mutation invariants of finite type}\

Fomin and Zelevinsky first classified all the finite type cluster algebras in \cite{FZ2}. In particular, the case of rank 2 is the cluster algebra with the initial exchange matrix
$$B=\begin{pmatrix}0 & m\\ -n & 0\end{pmatrix},\, 0\leq mn\leq 3,$$ 
see \cite[Theorem 1.8]{FZ2} and \cite[Theorem 5.1.1]{FWZ}. More precisely, $\mcA$ is of finite type if $mn\leq 3$, affine type if $mn=4$, and non-affine type if $mn\geq 5$, cf. \cite{N}.

\begin{example}[the $A_1 \times A_1$ type]\label{double A1} Let $m=n=0$. Then the mutation rules are
	$$\begin{array}{c} \mu_1(x_1,x_2)=(\dfrac{2}{x_1},x_2),\ \mu_2(x_1,x_2)=(x_1,\dfrac{2}{x_2}).\end{array}$$	
Hence, the clusters are 4-periodic and they are as follows
$$(x_1,x_2),\ (\frac{2}{x_1},x_2),\ (\frac{2}{x_1},\frac{2}{x_2}),\ (x_1,\frac{2}{x_2}).$$
Since the exchange matrices are all zero matrices (usually called \emph{isolated}), the characterization of Laurent mutation invariants of $A_1 \times A_1$ type is explicit.

\begin{lemma}\label{poly}
	Let $f(x)$ be a non-constant Laurent polynomial over $\mbQ$. If $f(x)=f(\dfrac{2}{x})$, there exists a polynomial $g(X)\in \mathbb{Q}[X]$ such that $f(x)=g(x+\dfrac{2}{x})$.
\end{lemma}
\begin{proof}
	Assume that $$f(x)=\dfrac{a_0+a_1x+\dots+a_mx^m}{x^k},$$ where $a_0\neq 0$. Since $f(x)=f(\dfrac{2}{x})$, we can directly get $k\geq 1$, $m\geq1$ and $a_m\neq 0$.
 Furthermore, it follows that $$2^kx^m(a_0+a_1x+\dots+a_mx^m)=x^{2k}(a_0x^m+2a_1x^{m-1}+\dots+2^ma_m).$$
	Comparing the highest degree of $x$ on both sides, we get $m=2k$ and $2^ka_{2k}=a_0$. Then $$f(x)=f_1(x)+\dfrac{a_1x+\dots+a_{2k-1}x^{2k-1}}{x^k},$$ where $f_1(x)=a_{2k}[x^k+(\dfrac{2}{x})^k]$. By induction, one can directly verify that there exists $F_1(X)\in \mathbb{Z}[X]$ such that  $f_1(x)=a_{2k}F_1(x+\dfrac{2}{x})$. Notice that $f(x)-f_1(x)$ is either a constant or a mutation invariant which can be simplified as $$\dfrac{a_1x+\dots+a_{2k-1}x^{2k-1}}{x^k}=\dfrac{a_s+\dots+a_{l}x^{l-s}}{x^{k-s}},$$ where $1\leq s<k$, $s<l\leq 2k-1<m$, $a_s\neq 0$ and $a_l\neq 0$. Since $l-s<m$, by repeating the process above, there exists a polynomial $g_1(X)\in \mathbb{Q}[X]$ such that $$\dfrac{a_s+\dots+a_{l}x^{l-s}}{x^{k-s}}=g_1(x+\dfrac{2}{x}).$$ Therefore, we get $g(X)=a_{2k}F_1(X)+g_1(X)\in \mathbb{Q}[X]$ and $f(x)=g(x+\dfrac{2}{x})$.
\end{proof}

\begin{proposition}\label{prop A1}
	A non-constant rational function $\mcT(x_1,x_2)$ is a Laurent mutation invariant of $A_1\times A_1$ type if and only if there exists a polynomial $G(X_1,X_2)\in \mathbb{Q}[X_1,X_2]$ such that $\mcT(x_1,x_2)=G(x_1+\dfrac{2}{x_1},x_2+\dfrac{2}{x_2}).$
\end{proposition}
\begin{proof}
	We begin with proving the sufficient condition and it is direct that $\mcT(x_1,x_2)=G(x_1+\dfrac{2}{x_1},x_2+\dfrac{2}{x_2})$ is a mutation invariant of $\mcA$ by calculation.
	
	Then, we prove the necessary condition and assume that $\mcT(x_1,x_2)$ is a Laurent invariant of $A_1\times A_1$ type. Fix $x_2$ and denote $\mcT(x_1,x_2)$ by $\mcT_{x_2}(x_1)\in \mbQ(x_2)(x_1)$. Notice that$$\mcT_{x_2}(x_1)=\mcT_{x_2}(\dfrac{2}{x_1}).$$ By \Cref{poly}, there exists a polynomial over the field $\mbQ(x_2)$, $$g_{x_2}(X)=A_0(x_2)+A_1(x_2)X+\dots+A_m(x_2)X^m,$$ such that $\mcT_{x_2}(x_1)=g_{x_2}(x_1+\dfrac{2}{x_1})$. Hence, it follows that $$\mcT(x_1,x_2)=A_0(x_2)+A_1(x_2)(x_1+\dfrac{2}{x_1})+\dots+A_m(x_2)(x_1+\dfrac{2}{x_1})^m.$$ Take $x_1=1, 3, \dots, 2m+1$ respectively and we get the following $m+1$ equations
	\begin{equation*}
		\begin{cases}
			\mcT(1,x_2)=A_0(x_2)+3A_1(x_2)+\dots+3^mA_m(x_2),\\
			\mcT(3,x_2)=A_0(x_2)+\dfrac{11}{3}A_1(x_2)+\dots+(\dfrac{11}{3})^mA_m(x_2),\\
			\;\;\;\;\;\;\,\vdots\\
			\mcT(2m+1,x_2)=A_0(x_2)+\dfrac{4m^2+4m+3}{2m+1}A_1(x_2)+\dots+(\dfrac{4m^2+4m+3}{2m+1})^mA_m(x_2).
		\end{cases}
	\end{equation*}
	Then we have 
	$$\begin{pmatrix} \mcT(1,x_2)\\ \mcT(3,x_2)\\ \vdots\\ \mcT(2m+1,x_2)\end{pmatrix}=\begin{pmatrix} 1 & 3 & \dots & 3^m\\ 1 & \dfrac{11}{3} & \dots & (\dfrac{11}{3})^m\\ \vdots & \vdots & \ddots & \vdots\\ 1 & \dfrac{4m^2+4m+3}{2m+1} & \dots & (\dfrac{4m^2+4m+3}{2m+1})^m\end{pmatrix}\begin{pmatrix} A_0(x_2)\\ A_1(x_2)\\ \vdots\\ A_m(x_2)\end{pmatrix}.$$ 
	Notice that the matrix 
	\begin{align}M=\begin{pmatrix} 1 & 3 & \dots & 3^m\\ 1 & \dfrac{11}{3} & \dots & (\dfrac{11}{3})^m\\ \vdots & \vdots & \ddots & \vdots\\ 1 & \dfrac{4m^2+4m+3}{2m+1} & \dots & (\dfrac{4m^2+4m+3}{2m+1})^m\end{pmatrix}\notag
	\end{align} is invertible since its determinant is Vandemonde determinant, which is nonzero. 
	Hence, \begin{align}\label{Vandemonde}\begin{pmatrix} A_0(x_2)\\ A_1(x_2)\\ \vdots\\ A_m(x_2)\end{pmatrix}=M^{-1}\begin{pmatrix} \mcT(1,x_2)\\ \mcT(3,x_2)\\ \vdots\\ \mcT(2m+1,x_2)\end{pmatrix}.\end{align}
	By \Cref{poly}, there are polynomials $g_i(X_2)\in \mbQ[X_2]\ (i=0,\dots,m)$, such that $$\mcT(2i+1,x_2)=g_i(x_2+\dfrac{2}{x_2}).$$ Furthermore, according to \eqref{Vandemonde}, there are $f_i(X_2)\in \mbQ[X_2]\ (i=0,\dots,m)$, such that $$A_i(x_2)=f_i(x_2+\dfrac{2}{x_2}).$$ Take the polynomial $G(X_1,X_2) \in \mathbb{Q}[X_1,X_2]$ as  
	$$G(X_1,X_2)=f_0(X_2)X_1+f_1(X_2)X_1+\dots+f_m(X_2)X_1^m,$$ and then we get $\mcT(x_1,x_2)=G(x_1+\dfrac{2}{x_1},x_2+\dfrac{2}{x_2})$.
\end{proof}
\begin{remark}\label{rmk A1xA1}
	It is clear that $$\mcT(x_1,x_2)=x_1+\dfrac{2}{x_1}+x_2+\dfrac{2}{x_2}$$ is a mutation invariant since we can take $$G(X_1,X_2)=X_1+X_2$$ in \Cref{prop A1}. Notice that the mutation invariant can also be expressed in another way as \Cref{M action provides mutation invariants rank 2}.
	Take the symmetric polynomial $$\Phi_1(X_1,\dots,X_4)=\dfrac{1}{4}(X_1+\dots +X_{4})$$ and the rational function $$F_1(X_1,X_2)=X_1+\dfrac{2}{X_1}+X_2+\dfrac{2}{X_2}.$$
	Then we get the mutation invariant $\mcT(x_1,x_2)$. On the other hand, we can also take the symmetric polynomial $$\Phi_2(X_1,\dots,X_4)=\dfrac{1}{2}(X_1+\dots +X_{4})$$ and the rational function $$F_2(X_1,X_2)=X_1+X_2 .$$
	Then we get the mutation invariant $\mcT(x_1,x_2)$ again. Hence, if we take different $\Phi$ and $F$, we may get a common mutation invariant.
\end{remark}
\end{example}
However, the good phenomenon (\Cref{prop A1}) is not suitable for the $A_2$, $B_2$, $G_2$ type. Hence, we transfer our attention back to the important examples of mutation invariants.
\begin{example}[the $A_2$ type]\label{A2} Let $m=n=1$. Then the mutation rules are
$$\begin{array}{c}
		\mu_{1}(x_1,x_2)=(\dfrac{x_2+1}{x_1},x_2),\ \mu_{2}(x_1,x_2)=(x_1,\dfrac{x_1+1}{x_2}).
	\end{array} $$
Note that the (labeled) clusters are 10-periodic and they are as follows
$$\begin{array}{c}
(x_1,x_2),\ (\frac{x_2+1}{x_1},x_2),\ (\frac{x_2+1}{x_1},\frac{x_1+x_2+1}{x_1x_2}),\ (\frac{x_1+1}{x_2},\frac{x_1+x_2+1}{x_1x_2}),\ (\frac{x_1+1}{x_2},x_1),\\
(x_2,x_1),\ (x_2,\frac{x_2+1}{x_1}),\ (\frac{x_1+x_2+1}{x_1x_2},\frac{x_2+1}{x_1}),\ (\frac{x_1+x_2+1}{x_1x_2},\frac{x_1+1}{x_2}),\ (x_1,\frac{x_1+1}{x_2}).
\end{array}$$
According to \Cref{M action provides mutation invariants rank 2}, we can get a mutation invariant of $A_2$ type as follows.
Take the symmetric polynomial $$\Phi(X_1,\cdots,X_{10})=\dfrac{1}{2}(X_1+\dots +X_{10}),$$ and rational function $$F(X_1,X_2)=X_1.$$ Then the mutation invariant is
	\begin{align}
	\mcT(x_1,x_2)=&{}x_1+x_2+\frac{x_2+1}{x_1}+\frac{x_1+x_2+1}{x_1x_2}+\frac{x_1+1}{x_2} \notag \\
	=&{}\frac{x_1^2x_2+x_1x_2^2+x_1^2+x_2^2+2x_1+2x_2+1}{x_1x_2}.\label{A2 invariant} 
\end{align}
\end{example}
\begin{example}[the $B_2$ type]\label{B2}
	Let $m=1$, $n=2$. Then the mutation rules are 
$$	\begin{array}{c}\mu_1(x_1,x_2)=(\dfrac{x^2_2+1}{x_1},x_2),\ \mu_2(x_1,x_2)=(x_1,\dfrac{x_1+1}{x_2}).
\end{array}$$
Note that the clusters are 6-periodic and they are as follows 
	$$\begin{array}{l}
		(x_1,x_2),\ (\frac{x_2^2+1}{x_1},x_2),\ (\frac{x_2^2+1}{x_1},\frac{x^2_2+x_1+1}{x_1x_2}),
		(\frac{x^2_2+x_1^2+2x_1+1}{x_1x_2^2},\frac{x^2_2+x_1+1}{x_1x_2}),\\(\frac{x^2_2+x_1^2+2x_1+1}{x_1x_2^2},\frac{x_1+1}{x_2}),\ (x_1,\frac{x_1+1}{x_2}).
	\end{array}$$
According to \Cref{M action provides mutation invariants rank 2}, we get mutation invariants of $B_2$ type as follows. Take the symmetric polynomial $$\Phi(X_1,\cdots,X_{6})=\dfrac{1}{2}(X_1+\dots +X_{6}),$$ and the rational function $$F_1(X_1,X_2)=X_1.$$ Then the mutation invariant is
	\begin{align}
		\mcT_1(x_1,x_2)=&{}x_1+\frac{x_2^2+1}{x_1}+\frac{x^2_2+x_1^2+2x_1+1}{x_1x_2^2} \notag \\
		=&{}\frac{x_1^2x_2^2+x_2^4+2x_2^2+x_1^2+2x_1+1}{x_1x_2^2}.\label{B2 invariant1} 
	\end{align}
Similarly, we can take another rational function $$F_2(X_1,X_2)=X_2$$ and get the mutation invariant
		\begin{align}
		\mcT_2(x_1,x_2)=&{}x_2+\frac{x_2^2+x_1+1}{x_1x_2}+\frac{x_1+1}{x_2} \notag \\
		=&{}\frac{x_1x_2^2+x_2^2+x_1^2+2x_1+1}{x_1x_2}.\label{B2 invariant2} 
	\end{align}
\end{example}
\begin{example}[the $G_2$ type]\label{G2}	
	Let $m=1$, $n=3$. Then the mutation rules are 
$$	\begin{array}{c}\mu_1(x_1,x_2)=(\dfrac{x_2^3+1}{x_1},x_2),\ \mu_2(x_1,x_2)=(x_1,\dfrac{x_1+1}{x_2}).\end{array}$$
Note that the clusters are 8-periodic and they are as follows
	$$\begin{array}{l}
		(x_1,x_2),\ (\frac{x_2^3+1}{x_1},x_2),\ (\frac{x_2^3+1}{x_1},\frac{x_2^3+x_1+1}{x_1x_2}),\ (\frac{x_2^6+3x_1x_2^3+2x_2^3+x_1^3+3x_1^2+3x_1+1}{x_1^2x_2^3},\frac{x_2^3+x_1+1}{x_1x_2}),\\(\frac{x_2^6+3x_1x_2^3+2x_2^3+x_1^3+3x_1^2+3x_1+1}{x_1^2x_2^3},\frac{x_2^3+x_1^2+2x_1+1}{x_1x_2^2}),\ (\frac{x_2^3+x_1^3+3x_1^2+3x_1+1}{x_1x_2^3},\frac{x_2^3+x_1^2+2x_1+1}{x_1x_2^2}),\\ (\frac{x_2^3+x_1^3+3x_1^2+3x_1+1}{x_1x_2^3},\frac{x_1+1}{x_2}),\ (x_1,\frac{x_1+1}{x_2}).
	\end{array}$$
According to \Cref{M action provides mutation invariants rank 2}, we get mutation invariants of $G_2$ type as follows. Take the symmetric polynomial $$\Phi(X_1,\cdots,X_{8})=\dfrac{1}{2}(X_1+\dots +X_{8}),$$ and the rational function $$F_1(X_1,X_2)=X_2.$$ Then the mutation invariant is
	\begin{align}
		\mcT_1(x_1,x_2)=&x_2+\frac{x^2_3+x_1+1}{x_1x_2}+\dfrac{x_2^3+x_1^2+2x_1+1}{x_1x_2^2}+\frac{x_1+1}{x_2} \notag \\
=&\frac{x_2^4+x_1x_2^3+x_2^3+x_1^2x_2+2x_1x_2+x_1^2+x_2+2x_1+1}{x_1x_2^2}.\label{G2 invariant1} 
	\end{align}
Similarly, we can take another rational function $$F_2(X_1,X_2)=X_1$$ and get the mutation invariant
 \begin{align}
	\mcT_2(x_1,x_2)=
\frac{x_1x_2^6+x_2^6+x_1^3x_2^3+5x_1x_2^3+x_1^4+2x_2^3+4x_1^3+6x_1^2+4x_1+1}{x_1^2x_2^3}. \label{G2 invariant2} 
 \end{align}
\end{example}

\subsection{Mutation invariants of affine type: Existence}\

For mutation invariants of finite type, there is a complete characterization as \Cref{M action provides mutation invariants rank 2}. However, for both affine and non-affine type, it does not hold. In the following, we provide two examples of mutation invariants of affine type,  that is $mn=4$ in \eqref{matrix}, thereby proving the existence.
\begin{example}[the $A_1^{(1)}$ type]\label{example1}
	Let $m=n=2$ and $\mathbf{x}=(x_{1}, x_{2})$ be the initial cluster. Then the mutation rules are $$\begin{array}{c}\mu_{1}(x_1,x_2)=(\dfrac{x_2^2+1}{x_1},x_2),\ \mu_{2}(x_1,x_2)=(x_1,\dfrac{x_1^2+1}{x_2}).
	\end{array}$$
	The following lemma exhibits a mutation invariant of $\mcA$.
	\begin{lemma}\label{AA} 
		Let $\mathcal{T}(x_1,x_2)$ be a rational function in $\mathbb{Q}(x_1,x_2)$ 
		defined by \begin{align}\mathcal{T}(x_1,x_2)=\dfrac{x_1^2+x_2^2+1}{x_1x_2}.\label{22 invariant}
		\end{align} 
	Then $\mathcal{T}( \mu_{i}(x_1,x_2))=\mathcal{T}(x_1,x_2)$, 
		$i= 1,2$, that is  $\mathcal{T}(x_1,x_2)$ is a mutation invariant.
	\end{lemma}
	\begin{proof}
		By the symmetry, without loss of generality, we only consider the case that $i=1$. Let $\mathcal{T}=\mathcal{T}(x_1,x_2)$ and we have $x_1^{2}+x_2^{2}+1=\mathcal{T}x_1x_2$. Hence $x_1$ can be viewed as a zero point of the polynomial $f=\lambda^{2}-x_2\mathcal{T}\lambda+x_2^{2}+1$. According to Vieta's formulas, $x_1^{\prime}=\frac{x_2^{2}+1}{x_1}$ is the other zero point of $f$, that is $\mathcal{T}( \mu_{1}(x_1,x_2))=\mathcal{T}(x_1,x_2)$. Hence $\mathcal{T}(x_1,x_2)$ is a mutation invariant. 
	\end{proof}
\begin{remark}\label{Integral}
The Vieta's formula $x_1x_1^{\prime}=x_2^2+1$ can be viewed as the cluster mutation rule $\mu_{1}$. Furthermore, there is another Vieta's formula $x_1+x_1^{\prime}=x_2\mathcal{T}$. Specifically, for any $x_1,x_2,\mathcal{T}\in \mathbb{N}_{+}$, it follows that $x_1^{\prime}=\frac{x_2^{2}+1}{x_1}\in \mathbb{N}_{+}$. Similarly, $x_2^{\prime}=\frac{x_1^{2}+1}{x_2}\in \mathbb{N}_{+}$.
\end{remark}
\end{example}
\begin{example}[the $A_2^{(2)}$ type]\label{example2} 
	Let $m=1,n=4$ and $\mathbf{x}=(x_{1}, x_{2})$ be the initial cluster. Then the mutation rules are
	$$\begin{array}{c}\mu_{1}(x_1,x_2)=(\dfrac{x_2^{4}+1}{x_1},x_2),\ \mu_{2}(x_1,x_2)=(x_1,\dfrac{x_1+1}{x_2}).
	\end{array}$$
	The following lemma exhibits a  mutation invariant of $\mcA$.
	\begin{lemma}\label{14}
		Let $\mathcal{T}(x_1,x_2)$ be a rational function in $\mathbb{Q}(x_1,x_2)$ defined by
	\begin{align} \mathcal{T}(x_1,x_2)=\dfrac{x_2^4+x_1^2+2x_1+1}{x_1x_2^2}.\label{14 invariant} \end{align} 
	Then $\mathcal{T}( \mu_{i}(x_1,x_2))=\mathcal{T}(x_1,x_2)$, 
		$i= 1,2$, that is  $\mathcal{T}(x_1,x_2)$ is a mutation invariant.
	\end{lemma}
	\begin{proof}
		Let $\mathcal{T}=\mathcal{T}(x_1,x_2)$. For $i=1$, $x_1$ can be viewed as a zero point of the quadratic polynomial $f_{1}=\lambda^{2}+(2-x_2^2\mathcal{T})\lambda+x_2^{4}+1$. According to Vieta's formulas, $x_1^{\prime}=\frac{x_2^4+1}{x_1}$ is another zero point of $f_{1}$ such that $\mathcal{T}( \mu_{1}(x_1,x_2))=\mathcal{T}(x_1,x_2)$.
		Similarly, for $i=2$, $x_2$ can be viewed as a zero point of the biquadratic polynomial $f_2=\lambda^4-x_1\mathcal{T}\lambda^2+(x_1+1)^2.$ By Vieta's formulas, $x_2^{\prime}=\frac{x_1+1}{x_2}$ is another zero point of $f_{2} $ such that $\mathcal{T}( \mu_{2}(x_1,x_2))=\mathcal{T}(x_1,x_2)$.
	\end{proof}
	
	\begin{remark}\label{int2}
		By the proof of \Cref{14}, we observe that the Vieta's formulas $x_1x_1' = x_2^4+1$ and $x_1 + x_1' = \mcT x_2^2-2$ can be identified as the mutation rules. Therefore, if $x_1,x_2,\mathcal{T}\in \mathbb{N}_{+}$, we can conclude that $x_1^{\prime}=\frac{x_2^{4}+1}{x_1}\in \mathbb{N}_{+}$. Furthermore, according to $x_2^2+(\frac{x_1+1}{x_2})^2=\mathcal{T}x_1$, we also have $x_2^{\prime}=\frac{x_1+1}{x_2}\in \mathbb{N}_{+}$. 
	\end{remark}
\begin{remark}
	The mutation invariants given by \Cref{AA} and \Cref{14} are both Laurent mutation invariants.
	\end{remark}
\end{example}
\begin{question}
	What is the characterization of mutation invariants of the cluster algebra $\mcA$ of affine type?
\end{question}
\subsection{Laurent mutation invariants of non-affine type: Non-existence}\

We have proved that for \eqref{matrix}, when $mn\leq4$, there exist mutation invariants of $\mcA$. In this subsection, we aim to prove that when $mn\geq 5$, that is of non-affine type, there does not exist a Laurent mutation invariant.

Firstly, we introduce $\mfd$-vectors of rank 2 and provide a different proof of the established theorems \cite[Theorem 1.8]{FZ2} or \cite[Theorem 5.1.1]{FWZ}. When $n=2$, the 2-regular tree $\mbT_{2}$ indexing the cluster pattern is denoted by
\begin{align}\dots\stackrel{2}{\longleftrightarrow}t_{-2}\stackrel{1}{\longleftrightarrow}t_{-1}\stackrel{2}{\longleftrightarrow}t_0 \stackrel{1}{\longleftrightarrow}t_1\stackrel{2}{\longleftrightarrow}t_2\stackrel{1}{\longleftrightarrow} \dots \label{T2}
\end{align}
and the recurrence relations for $\mfd$-vectors as \eqref{d-vector} can be reduced to
\begin{align} \notag
	\mfd_{l;t_{1}}=\left\{
	\begin{array}{ll}
		-\mfd_{1;t_{0}}, &   l=1,\\
		\mfd_{2;t_{0}}, &   l= 2,
	\end{array} \right.
	\mfd_{l;t_{2k+1}}=\left\{
	\begin{array}{ll}
		-\mfd_{1;t_{2k}}+n\mfd_{2;t_{2k}}, &   l=1,\\
		\mfd_{2;t_{2k}}, &   l= 2,
	\end{array} \right.
\end{align}
respectively for $t_0\stackrel{1}{\longleftrightarrow}t_{1}$, $t_{2k}\stackrel{1}{\longleftrightarrow}t_{2k+1}(k\neq 0)$, and
\begin{align} \notag
	\mfd_{l;t_{0}}=\left\{
	\begin{array}{ll}
		\mfd_{1;t_{-1}}, &   l= 1, \\
		-\mfd_{2;t_{-1}}, &   l=2,
	\end{array} \right.
	\mfd_{l;t_{2k}}=\left\{
	\begin{array}{ll}
		\mfd_{1;t_{2k-1}}, &   l=1, \\
		-\mfd_{2;t_{2k-1}}+m\mfd_{1;t_{2k-1}}, &   l=2,
	\end{array} \right.
\end{align}
respectively for $t_{-1}\stackrel{2}{\longleftrightarrow}t_{0}$, $t_{2k-1}\stackrel{2}{\longleftrightarrow}t_{2k}(k\neq 0)$. It is clear that 
$$\mfd_{1;t_1}=\begin{pmatrix}
	1 \\ 0
\end{pmatrix}, \mfd_{2;t_1}=\begin{pmatrix}
	0 \\ -1
\end{pmatrix}, \mfd_{1;t_2}=\begin{pmatrix}
	1 \\ 0
\end{pmatrix}, \mfd_{2;t_2}=\begin{pmatrix}
	m \\ 1
\end{pmatrix}.$$ 

Let $U$ and $V$ be matices in the following $$U=\begin{pmatrix}1 & 0 & 0 & 0\\ m & -1& 0 & 0\\ 0 & 0 & -1 & n\\ 0 & 0 & 0 & 1\end{pmatrix}, V=\begin{pmatrix}-1 & n & 0 & 0\\ 0 & 1& 0 & 0\\ 0 & 0 & 1 & 0\\ 0 & 0 & m & -1\end{pmatrix},$$ and denote the formal vectors by
 $$B_{k}=\begin{pmatrix}\mfd_{1;t_{2k}} \\ \mfd_{2;t_{2k}} \\ \mfd_{1;t_{2k+1}} \\ \mfd_{2;t_{2k+1}}\end{pmatrix}, A_{k}=\begin{pmatrix}\mfd_{1;t_{2k-1}} \\ \mfd_{2;t_{2k-1}} \\ \mfd_{1;t_{2k}} \\ \mfd_{2;t_{2k}}\end{pmatrix}.$$ Hence, we get $$W=UV=\begin{pmatrix}-1 & n & 0 & 0\\ -m & mn-1& 0 & 0\\ 0 & 0 & mn-1 & -n\\ 0 & 0 & m & -1\end{pmatrix},$$ and for any $k\geq 1$, $$B_{k}=WB_{k-1}=W^{k-1}B_{1}=W^{k-1}UA_{1}.$$
 
Note that when $mn=4$, $W$ is not diagonalizable but uptriangularizable. When $mn\geq 5$, $W$ is diagonalizable. Hence, the expressions of $\mfd$-vectors of rank 2 with $mn\geq 4$ are as follows. 
\begin{lemma}\label{d}
	There are three cases of expressions of $\mfd$-vectors of rank 2 with $mn\geq 4$.
	\begin{enumerate}
		\item Case that $m=n=2$: for any $k\geq 1$,\\
		$\begin{array}{l} \mfd_{1;t_{2k}}=(2k-1)\mfd_{1;t_{1}}-(2k-2)\mfd_{2;t_{1}},\\
			\mfd_{2;t_{2k}}=2k\mfd_{1;t_{1}}-(2k-1)\mfd_{2;t_{1}},\\
			\mfd_{1;t_{2k+1}}=-(2k-1)\mfd_{1;t_{2}}+2k\mfd_{2;t_{2}},\\
			\mfd_{2;t_{2k+1}}=-(2k-2)\mfd_{1;t_{2}}+(2k-1)\mfd_{2;t_{2}}.
		\end{array}$
		\item Case that $m=1$, $n=4$: for any $k\geq 1$,\\
		$\begin{array}{l}
			\mfd_{1;t_{2k}}=(2k-1)\mfd_{1;t_{1}}-(4k-4)\mfd_{2;t_{1}},\\
			\mfd_{2;t_{2k}}=k\mfd_{1;t_{1}}-(2k-1)\mfd_{2;t_{1}},\\
			\mfd_{1;t_{2k+1}}=-(2k-1)\mfd_{1;t_{2}}+4k\mfd_{2;t_{2}},\\
			\mfd_{2;t_{2k+1}}=-(k-1)\mfd_{1;t_{2}}+(2k-1)\mfd_{2;t_{2}}.
		\end{array}$
		\item Case that $mn\geq 5$: for any $k\geq 1$,\\
		$\begin{array}{l}
			\mfd_{1;t_{2k}}=(\alpha_{1,k-1}+m\alpha_{2,k-1})\mfd_{1;t_{1}}-\alpha_{2,k-1}\mfd_{2;t_{1}},\\
			\mfd_{2;t_{2k}}=(\alpha_{3,k-1}+m\alpha_{4,k-1})\mfd_{1;t_{1}}-\alpha_{4,k-1}\mfd_{2;t_{1}},\\
			\mfd_{1;t_{2k+1}}=-\beta_{1,k-1}\mfd_{1;t_{2}}+(n\beta_{1,k-1}+\beta_{2,k-1})\mfd_{2;t_{2}},\\
			\mfd_{2;t_{2k+1}}=-\beta_{3,k-1}\mfd_{1;t_{2}}+(n\beta_{3,k-1}+\beta_{4,k-1})\mfd_{2;t_{2}},\\
		\end{array}$\\
		\text{where} $a= -1+\dfrac{mn}{2}$, $b= \dfrac{\sqrt{mn(mn-4)}}{2}$ \text{and} \\
		$\begin{array}{l}
			\alpha_{1,k-1}=\dfrac{1}{2}[(a-b)^{k-1}+(a+b)^{k-1}+\dfrac{\sqrt{mn}(a-b)^{k-1}-\sqrt{mn}(a+b)^{k-1}}{\sqrt{mn-4}}],\\
			\alpha_{2,k-1}=\dfrac{-\sqrt{n}(a-b)^{k-1}+\sqrt{n}(a+b)^{k-1}}{\sqrt{m(mn-4)}},\\
			\alpha_{3,k-1}=\dfrac{\sqrt{m}(a-b)^{k-1}-\sqrt{m}(a+b)^{k-1}}{\sqrt{n(mn-4)}},\\
			\alpha_{4,k-1}=\dfrac{(-\sqrt{mn}+\sqrt{mn-4})(a-b)^{k-1}+(\sqrt{mn}+\sqrt{mn-4})(a+b)^{k-1}}{2\sqrt{mn-4}},\\
			\beta_{1,k-1}=\dfrac{1}{2}[(a-b)^{k-1}+(a+b)^{k-1}+\dfrac{-\sqrt{mn}(a-b)^{k-1}+\sqrt{mn}(a+b)^{k-1}}{\sqrt{mn-4}}],\\
			\beta_{2,k-1}=\dfrac{\sqrt{n}(a-b)^{k-1}-\sqrt{n}(a+b)^{k-1}}{\sqrt{m(mn-4)}},\\
			\beta_{3,k-1}=\dfrac{-\sqrt{m}(a-b)^{k-1}+\sqrt{m}(a+b)^{k-1}}{\sqrt{n(mn-4)}},\\
			\beta_{4,k-1}=\dfrac{(\sqrt{mn}+\sqrt{mn-4})(a-b)^{k-1}+(-\sqrt{mn}+\sqrt{mn-4})(a+b)^{k-1}}{2\sqrt{mn-4}}.
		\end{array}$
	\end{enumerate}
\end{lemma}
\begin{lemma}\label{infinite}
	Each component of $\mfd_{i;t_j}$ with $mn\geq 4$ tends to positive infinity as $j$ tends to positive infinity. 
\end{lemma}
\begin{proof} Firstly, by \Cref{d}, it is direct that the lemma holds for $mn=4$. For $mn\geq 5$, without loss of generality, we only consider $\mfd_{1;t_{2k}}$ since other cases are similar. The first component of $\mfd_{1;t_{2k}}$ is 
	$$\alpha_{1,k-1}+m\alpha_{2,k-1}=\dfrac{1}{2}[(a-b)^{k-1}+(a+b)^{k-1}+\dfrac{\sqrt{mn}(a+b)^{k-1}-\sqrt{mn}(a-b)^{k-1}}{\sqrt{mn-4}}].$$
	Notice that $a>b>1$ and $0<a-b<1$. Hence, it tends to positive infinity as $k$ tends to positive infinity.
	The second component of    $\mfd_{1;t_{2k}}$ is 
	$$\alpha_{2,k-1}=\dfrac{-\sqrt{n}(a-b)^{k-1}+\sqrt{n}(a+b)^{k-1}}{\sqrt{m(mn-4)}}.$$ It also tends to positive infinity as $k$ tends to positive infinity.
\end{proof}

\begin{theorem}[{\cite[Theorem 1.8]{FZ2}}, {\cite[Theorem 5.1.1]{FWZ}}]
	A cluster algebra $\mcA$ of rank 2 with the initial exchange matrix \eqref{matrix} is of finite type if and only if $mn \leq 3$.
\end{theorem}
\begin{proof}
	When $mn\leq3$, it is clear that the cluster algebra $\mcA$ is of finite type by direct calculation. When $mn\geq 4$, by \Cref{infinite}, the cluster algebra $\mcA$ is of infinite type.
\end{proof}
From now on, we focus on proving the non-existence of Laurent mutation invariant of non-affine type. First of all, we need several preparatory lemmas in the following.
\begin{lemma}\label{term 1} For any cluster algebra $\mcA$ of rank 2 and $i\in \{1,2\}$, the numerator $N_{i;t_k}(x_1, x_2)$ in \eqref{express of d-vector} has constant term $1$ with $|k|\geq 2$.
\end{lemma}
\begin{proof}
	By symmetry, without loss of generality, we can assume $k\geq0$ and take the induction on $k\geq 2$. When $k=2$, according to \eqref{T2}, it is easy to check that 
	 $$N_{1;t_2}(x_1, x_2)=x_2^n+1\ \text{and}\ N_{2;t_2}(x_1, x_2)=(x_2^n+1)^m+x_1^m, $$ and they both have constant term 1. Assume that the lemma holds for $k=h$. Without loss of generality, we can assume that $h$ is odd. When $k=h+1$, according to the the assumption and 2-direction mutation at $t_h$, we get $$N_{1;t_{h+1}}(x_1, x_2)=N_{1;t_{h}}(x_1, x_2),$$ which has constant term 1. Furthermore, by cluster mutation rules, we get 
	 \begin{align}\dfrac{N_{2;t_{h+1}}(x_1, x_2)}{x_1^{d_{12;t_{h+1}}} x_2^{d_{22;t_{h+1}}}}&=\dfrac{(\frac{N_{1;t_{h}}(x_1, x_2)}{x_1^{d_{11;t_{h}}} x_2^{d_{21;t_{h}}}})^m+1}{\frac{N_{2;t_{h}}(x_1, x_2)}{x_1^{d_{12;t_{h}}} x_2^{d_{22;t_{h}}}}}\notag \\ &= \dfrac{N_{1;t_{h}}^m(x_1, x_2)+x_1^{md_{11;t_{h}}} x_2^{md_{21;t_{h}}}}{N_{2;t_{h}}(x_1, x_2)x_1^{md_{11;t_{h}}-d_{12;t_h}} x_2^{md_{21;t_{h}}-d_{22;t_h}}}.\notag
	 \end{align}
Notice that $x_1,x_2$ do not divide $N_{1;t_{h}}(x_1,x_2)$, and both $N_{1;t_{h}}(x_1, x_2)$ and  $N_{2;t_{h}}(x_1, x_2)$ have constant term 1, we get  
$$N_{2;t_{h}}(x_1,x_2)N_{2;t_{h+1}}(x_1,x_2)= N_{1;t_{h}}^m(x_1, x_2)+x_1^{md_{11;t_{h}}} x_2^{md_{21;t_{h}}}.$$ Subsequently, we get $N_{2;t_{h+1}}(x_1,x_2)$ has constant term 1 and the lemma holds.
\end{proof}
In the following, for a Laurent polynomial $\mcT(x_1,x_2)\in \mathbb{Q}[x_1^{\pm1},x_2^{\pm1}]$, we denote it by 
\begin{align}
	\mcT(x_1,x_2)=\dfrac{\sum_{i,j}\lambda_{ij}x_1^{i}x_2^{j}}{x_1^{s}x_2^{t}}. \label{LP}
\end{align} Note that there is a necessary condition for $\mcT(x_1,x_2)$ to be a Laurent mutation invariant.
\begin{lemma}\label{degree twice}
	If $\mcT(x_1,x_2)$ is a Laurent mutation invariant of $\mcA$ with the initial exchange matrix \eqref{matrix}, the highest degree of $x_1$ and $x_2$ in the numerator \eqref{LP} is $2s$ and $2t$ respectively.
\end{lemma}
\begin{proof}
	Assume that $\mcT(x_1,x_2)$ is a Laurent mutation invariant of $\mcA$ as \eqref{LP}. It is direct that $s\geq 1$ and $t\geq 1$. Since $\mcT(x_1,x_2)$ is not a constant, there must be items $\{x_1^{i_0}x_2^{j_1},\dots,x_1^{i_0}x_2^{j_r}\}$ in the numerator which are not  $x_1^{s}x_2^{t}$ such that $\lambda_{i_0,j_k}\neq 0$, where $k\in\{1,\dots,r\}$ and $i_0$ is maximal. Notice that $\mcT(\mu_1(x_1,x_2))=\mcT(x_1,x_2)$, we have 
	\begin{align}\dfrac{\sum_{i,j}\lambda_{ij}(\frac{x_2^n+1}{x_1})^{i}x_2^{j}}{(\frac{x_2^n+1}{x_1})^{s}x_2^{t}}=\dfrac{\sum_{i,j}\lambda_{ij}x_1^{i}x_2^{j}}{x_1^{s}x_2^{t}}.\label{lem2.34.1}	
	\end{align}
 Multiplying $x_1^{i_0}$ to the numerator and denominator on the left hand of \eqref{lem2.34.1} simutaneously, we get \begin{align}\dfrac{\sum_{i,j}\lambda_{ij}x_1^{i_0-i}(x_2^n+1)^ix_2^j}{x_1^{i_0-s}(x_2^n+1)^sx_2^t}=\dfrac{\sum_{i,j}\lambda_{ij}x_1^{i}x_2^{j}}{x_1^{s}x_2^{t}}. \label{lem2.34.2}\end{align} Notice that in the numerator on the left hand side of \eqref{lem2.34.2}, there are nonzero items concerning $x_2$ which are not divisible by $x_1$ as follows $$\lambda_{i_0,j_1}(x_2^n+1)^{i_0}x_2^{j_1}+\dots+\lambda_{i_0,j_r}(x_2^n+1)^{i_0}x_2^{j_r}.$$ Hence $x_1$ does not divide  $\sum_{i,j}\lambda_{ij}x_1^{i_0-i}(x_2^n+1)^ix_2^j$ and we get $i_0-s=s$, which implies $i_0=2s$. Similarly, by $\mcT(\mu_2(x_1,x_2))=\mcT(x_1,x_2)$, the maximal degree of $x_2$ is $2t$.
\end{proof}
\begin{remark} 
	The examples can be referred to the Laurent mutation invariants \eqref{A2 invariant}, \eqref{B2 invariant1}, \eqref{B2 invariant2}, \eqref{G2 invariant1}, \eqref{G2 invariant2} of finite type and \eqref{14 invariant}, \eqref{22 invariant} of affine type.
\end{remark}
Now, in the following, we aim to prove our main theorem.
\begin{theorem}\label{Thm4.3}  There does not exist a Laurent mutation invariant of non-affine type.
\end{theorem}
\begin{proof}
	Assume that there is a non-constant Laurent mutation invariant $\mcT(x_1,x_2)$ of $\mcA$. By \Cref{degree twice}, we get  \begin{align}
		\mcT(x_1,x_2)=
		\dfrac{\sum_{i=0}^{2s}\sum_{j=0}^{2t}\lambda_{ij}x_1^{i}x_2^{j}}{x_1^{s}x_2^{t}},\label{T-express}
	\end{align}
where $s\geq 1$ and $t\geq 1$ are fixed positive integers. Without loss of generality, we can assume $\lambda_{st}=0$ since it corresponds to a conctant term of $\mcT(x_1,x_2)$. According to \Cref{d}, we get the $\mfd$-vectors for cluster variables $x_{1;t_{2k}}$ and $x_{2;t_{2k}}$ as 
	\begin{align}
		\mfd_{1;t_{2k}}=(\alpha_{1,k-1}+m\alpha_{2,k-1})\mfd_{1;t_{1}}-\alpha_{2,k-1}\mfd_{2;t_{1}}= \begin{pmatrix} \alpha_{1,k-1}+m\alpha_{2,k-1}\\ \alpha_{2,k-1}\end{pmatrix},\notag\\
		\mfd_{2;t_{2k}}=(\alpha_{3,k-1}+m\alpha_{4,k-1})\mfd_{1;t_{1}}-\alpha_{4,k-1}\mfd_{2;t_{1}}= \begin{pmatrix} \alpha_{3,k-1}+m\alpha_{4,k-1}\\ \alpha_{4,k-1}\end{pmatrix}.\notag
	\end{align}
Furthermore, according to \eqref{express of d-vector}, the cluster variables $x_{1;2k}$ and $x_{2;2k}$ can be expressed by the initial cluster varibales $x_1$, $x_2$ and the $\mfd$-vectors as 
\begin{align}
	x_{1;t_{2k}}=\dfrac{N_{1;t_{2k}}(x_1,x_2)}{x_1^{\alpha_{1,k-1}+m\alpha_{2,k-1}}x_2^{\alpha_{2,k-1}}},\  x_{2;t_{2k}}=\dfrac{N_{2;t_{2k}}(x_1,x_2)}{x_1^{\alpha_{3,k-1}+m\alpha_{4,k-1}}x_2^{\alpha_{4,k-1}}},\label{d-express}
\end{align}
where $N_{1;t_{2k}}(x_1,x_2)$ and $N_{2;t_{2k}}(x_1,x_2)\in \mathbb{Z}[x_1,x_2]$. Since $\mcT(x_1,x_2)=\mcT(x_{1;t_{2k}},x_{2;t_{2k}})$ for any $k\in \mbN$, by \eqref{T-express} and \eqref{d-express}, we get 
\begin{align}
\dfrac{\sum_{i=0}^{2s}\sum_{j=0}^{2t}\lambda_{ij}x_1^{i}x_2^{j}}{x_1^{s}x_2^{t}}=
	\dfrac{\sum_{i=0}^{2s}\sum_{j=0}^{2t}\lambda_{ij}\frac{N_{1;t_{2k}}^{i}(x_1,x_2)}{x_1^{i(\alpha_{1,k-1}+m\alpha_{2,k-1})}x_2^{i\alpha_{2,k-1}}}
	\frac{N_{2;t_{2k}}^{j}(x_1,x_2)}{x_1^{j(\alpha_{3,k-1}+m\alpha_{4,k-1})}x_2^{j\alpha_{4,k-1}}}}{\frac{N_{1;t_{2k}}^{s}(x_1,x_2)}{x_1^{s(\alpha_{1,k-1}+m\alpha_{2,k-1})}x_2^{s\alpha_{2,k-1}}}\frac{N_{2;t_{2k}}^{t}(x_1,x_2)}{x_1^{t(\alpha_{3,k-1}+m\alpha_{4,k-1})}x_2^{t\alpha_{4,k-1}}}}\label{big}.
\end{align}
 Denote that $$\begin{array}{c}
 	M_k=\max\{i(\alpha_{1,k-1}+m\alpha_{2,k-1})+j(\alpha_{3,k-1}+m\alpha_{4,k-1})| \lambda_{ij}\neq 0, k\gg 0 \},\\
 N_k=\max\{i\alpha_{2,k-1}+j\alpha_{4,k-1}| \lambda_{ij}\neq 0, k\gg 0 \}.
 \end{array}$$
Now we focus on the degree of $x_1$ on the right hand of \eqref{big}. Multiplying both the numerator and denominator by $x_1^{M_{k}}$ and $x_2^{N_k}$, we get \eqref{big} equals to the following
\begin{align}
	\dfrac{\sum_{i=0}^{2s}\sum_{j=0}^{2t}\lambda_{ij}N_{1;t_{2k}}^{i}(x_1,x_2)N_{2;t_{2k}}^{j}(x_1,x_2)x_1^{M_k-i(\alpha_{1,k-1}+m\alpha_{2,k-1})-j(\alpha_{3,k-1}+m\alpha_{4,k-1})}x_2^{N_k-i\alpha_{2,k-1}-j\alpha_{4,k-1}}}{N_{1;t_{2k}}^{s}(x_1,x_2)N_{2;t_{2k}}^{t}(x_1,x_2)x_1^{M_k-s(\alpha_{1,k-1}+m\alpha_{2,k-1})-t(\alpha_{3,k-1}+m\alpha_{4,k-1})}x_2^{N_k-s\alpha_{2,k-1}-t\alpha_{4,k-1}}}.\label{long}
\end{align}
According to \Cref{term 1}, we get both $N_{1;t_{2k}}(x_1,x_2)$ and $N_{2;t_{2k}}(x_1,x_2)$ have constant term 1, and $x_1$ does not divide the numerator of \eqref{long}. In addition, the denominator of \eqref{long} must be a polynomial, which means that the degree of monomials about $x_1$ and $x_2$ is positive. Notice that both $x_1$ and $x_2$ do not divide $N_{1;t_{2k}}(x_1,x_2)$ and $N_{2;t_{2k}}(x_1,x_2)$, we obtain that  $N_{1;t_{2k}}^{s}(x_1,x_2)N_{2;t_{2k}}^{t}(x_1,x_2)$ divides the numerator of \eqref{long} and 
\begin{align}M_k-s(\alpha_{1,k-1}+m\alpha_{2,k-1})-t(\alpha_{3,k-1}+m\alpha_{4,k-1})= s, \label{equiv 0}
\end{align} for any $k\gg 0$. Assume that $(i,j)=(i_0,j_0)$ for $M_k$. Hence, we get 
\begin{align}
	u(a-b)^{k-1}+v(a+b)^{k-1}= s, \label{equiv}
\end{align}
where $s\geq 1$ is a fixed positive integer and 
$$\begin{array}{c}
	u=\dfrac{(i_0-s)[\sqrt{n(mn-4)}-n\sqrt{m}]+(j_0-t)[2\sqrt{m}-mn\sqrt{m}+m\sqrt{n(mn-4)}]}{2\sqrt{n(mn-4)}},\\
	
	v=\dfrac{(i_0-s)[\sqrt{n(mn-4)}+n\sqrt{m}]+(j_0-t)[2\sqrt{m}+mn\sqrt{m}+m\sqrt{n(mn-4)}]}{2\sqrt{n(mn-4)}}.
\end{array}$$
However, by \Cref{infinite}, it is clear that 
\begin{align}
	\lim_{k\to +\infty}(a-b)^{k-1}= 0,\ \lim_{k\to +\infty}(a+b)^{k-1}= +\infty,\notag
\end{align}
which contradict with \eqref{equiv} regardless of the value of $u$ and $v$. Hence, we proved the non-existence of Laurent mutation invariant $\mcT(x_1,x_2)$.
\end{proof}
\begin{remark}
There are differences between the affine type with $mn=4$ and non-affine type with $mn\geq 5$ as follows.
\begin{enumerate} 
		\item When $m=n=2$, the equality \eqref{equiv 0} is \begin{align}(2k-1)(i_0-s)+2k(j_0-t)= s,\ \text{for any}\ k\gg 0. \label{eq1}\end{align}
		\item When $m=1,n=4$, the equality \eqref{equiv 0} is \begin{align}(2k-1)(i_0-s)+k(j_0-t)= s,\ \text{for any}\ k\gg 0.\label{eq2}\end{align}
	\end{enumerate}
Both equalities \eqref{eq1} and \eqref{eq2} can hold for certain $i_0, j_0, s, t$, see \Cref{example1} and \Cref{example2}. However, the equality \eqref{equiv 0} does not hold regardless of the choices of $i_0, j_0, s, t$.
\end{remark}
\begin{question}
	In \Cref{Thm4.3}, we have proved that there does not exist a Laurent mutation invariant of $\mcA$. Hence, a natural question is that is there a mutation invariant of $\mcA$ (i.e. in $\mathbb{Q}(x_1,x_2)\backslash \mathbb{Q}[x_1^{\pm1},x_2^{\pm1}]$) ?
\end{question}
\section{Applications: Diophantine equations characterized by mutations}

In this section, as an application of mutation invariants, we exhibit the Diophantine equations encoded with cluster algebras, which can be characterized by the initial solution and cluster mutations.
\subsection{Diophantine equations of type $A_1\times A_1$}\

By \Cref{M action provides mutation invariants rank 2}, the Diophantine equations about $x_1$ and $x_2$ encoded with the cluster algebra of type $A_1\times A_1$ are as follows
\begin{align}\Phi(F(x_1,x_2),F(\dfrac{2}{x_1},x_2),F(\dfrac{2}{x_1},\dfrac{2}{x_2}),F(x_1,\dfrac{2}{x_2}))=\mcT(a,b), \label{eqA1}
\end{align}	
 for any symmetric polynomial $\Phi(X_1,X_2,X_3,X_4)$, rational function $F(X_1,X_2)$ and $a,b\in \mbN$. Notice that the equation \eqref{eqA1} with rational coefficients can always be adjusted to Diophantine equations with integer coefficients. In addition, $(a,b)$ must be a solution to \eqref{eqA1} which is called the \emph{initial solution}. In particular, as \Cref{rmk A1xA1}, take $(a,b)=(1,1)$, we can solve a Diophantine equation as follows. 
\begin{lemma}\label{6}
	For the Diophantine equation with two variables as follows
	\begin{align}x_1^2x_2+x_1x_2^2+2x_1+2x_2=6x_1x_2,
	\label{dio A1xA1} \end{align} all the positive integer solutions can be derived from the initial solution through a finite number of mutations of type $A_1\times A_1$.
\end{lemma}
\begin{proof}
	First of all, notice that the Diophantine equation is equivalent to \begin{align}x_1+\dfrac{2}{x_1}+x_2+\dfrac{2}{x_2}=6. \label{dio A1}\end{align}
Since $(1,1)$ is an initial solution, by \Cref{rmk A1xA1}, we obtain a sequence of solutions to \eqref{dio A1} through mutations as follows
	$$(1,1) \stackrel{\mu_{1}}{\longleftrightarrow} (2,1) \stackrel{\mu_{2}}{\longleftrightarrow} (2,2)\stackrel{\mu_{1}}{\longleftrightarrow} (1,2)\stackrel{\mu_{2}}{\longleftrightarrow}(1,1).$$ Now, we claim that there are no other positive integer solutions apart from the four mentioned above. Assume that $(a,b)$ is another solution. If $\frac{2}{a}$ or $\frac{2}{b}$ is a positive integer, then $a+\frac{2}{a}=3$ and $b+\frac{2}{b}=3$, implying that $(a,b)$ is one of the four solutions above. Hence, we get $a\geq 3$ and $b\geq 3$ which contradict with \eqref{dio A1}. Therefore, all the positive integer solutions to \eqref{dio A1xA1} can be derived from the initial solution $(1,1)$ through a finite number of cluster mutations.
\end{proof}
\begin{remark}
	In fact, one can solve a more complicated Diophantine equation as follows in the similar method: $wxyz(w+x-y-z)=2(wxy+wxz-wyz-xyz)$. \end{remark}

\subsection{Diophantine equations of type $A_2$}\

By \Cref{M action provides mutation invariants rank 2}, the Diophantine equations about $x_1$ and $x_2$ encoded with the cluster algebra of type $A_2$ are as follows
\begin{align}\Phi(F(c_{1;1}(x_1,x_2),c_{2;1}(x_1,x_2)),\cdots,F(c_{1;10}(x_1,x_2),c_{2;10}(x_1,x_2)))=\mcT(a,b), \label{eqA2}
\end{align}	
 for any symmetric polynomial $\Phi(X_1,\dots,X_{10})$, rational function $F(X_1,X_2)$ and $a,b\in \mbN$. Note that $(c_{1;i}(x_1,x_2),c_{2;i}(x_1,x_2))^{10}_{i=1}$ are 10 different clusters of type $A_2$. It is clear that $(a,b)$ is a solution to \eqref{eqA2} which is called the \emph{initial solution}. In particular, in \Cref{A2}, take $(a,b)=(1,1)$ and we can solve a Diophantine equation as follows. 
\begin{lemma}\label{9}
	For the Diophantine equation with two variables as follows
	\begin{align}x_{1}^{2}x_{2}+x_{1}x_{2}^{2}+x_{1}^2+x_{2}^2+2x_{1}+2x_{2}+1=9x_{1}x_{2}, \label{dio A2} \end{align}all the positive integer solutions can be derived from the initial solution through a finite number of mutations of type $A_2$.
\end{lemma}
\begin{proof}
	Since $(1,1)$ is an initial solution, by \Cref{A2}, we obtain a sequence of solutions to \eqref{A2} through mutations as follows
	\begin{align}   &(1,1) \stackrel{\mu_{2}}{\longleftrightarrow} (1,2) \stackrel{\mu_{1}}{\longleftrightarrow} (3,2)\stackrel{\mu_{2}}{\longleftrightarrow} (3,2)\stackrel{\mu_{1}}{\longleftrightarrow} (1,2)
	\stackrel{\mu_{2}}{\longleftrightarrow} (1,1)\notag \\& \stackrel{\mu_{1}}{\longleftrightarrow} (2,1) \stackrel{\mu_{2}}{\longleftrightarrow} (2,3) \stackrel{\mu_{1}}{\longleftrightarrow} (2,3)  \stackrel{\mu_{2}}{\longleftrightarrow} (2,1) \stackrel{\mu_{1}}{\longleftrightarrow} (1,1) 
	.\notag
	\end{align}
	 Now, we claim that there are not other solutions. Assume that $(a,b)$ is another positive integer solution to \eqref{dio A2}. By \Cref{A2}, we obtain that $$(a',b)=\mu_{1}(a,b)=(\dfrac{b+1}{a},b)$$ is also a solution to \eqref{dio A2}. Notice that \begin{align}
		(b+1)a^2+(b^2-9b+2)a+(b+1)^2=0.\notag
	\end{align} By Vieta's formulas, we get $$a+a'=\dfrac{-b^{2}+9b-2}{b+1},\ aa'=b+1.$$ If $b\geq4$, we observe that $$a+a'\geq 2\sqrt{aa'}=2\sqrt{b+1}.$$ However, by monotonicity, it is direct that $$\dfrac{-b^{2}+9b-2}{b+1}- 2\sqrt{b+1}\leq \dfrac{18}{5}-2\sqrt{5}<0,$$ which is a contradiction. Hence, we conclude that $b<4$. Based on the results above, the listed solutions are complete. In other words, all the positive integer solutions to \eqref{dio A2} can be derived from the initial solution $(1,1)$ through a finite number of mutations.

\end{proof}
\begin{remark}
	Note that \Cref{9} is also proved in \cite[Proposition 23]{GM} from a different view.
\end{remark}
\subsection{Diophantine equations of type $B_2$}\

By \Cref{M action provides mutation invariants rank 2}, the Diophantine equations about $x_1$ and $x_2$ encoded with the cluster algebra of type $B_2$ are as follows
\begin{align}\Phi(F(c_{1;1}(x_1,x_2),c_{2;1}(x_1,x_2)),\cdots,F(c_{1;6}(x_1,x_2),c_{2;6}(x_1,x_2)))=\mcT(a,b), \label{eqB2}
\end{align}	
 for any symmetric polynomial $\Phi(X_1,\dots,X_{6})$, rational function $F(X_1,X_2)$ and $a,b\in \mbN$. Note that $(c_{1;i}(x_1,x_2),c_{2;i}(x_1,x_2))^6_{i=1}$ are 6 different clusters of type $B_2$. It is clear that $(a,b)$ is a solution to \eqref{eqB2} which is called the \emph{initial solution}. In particular, as \eqref{B2 invariant1} in \Cref{B2}, take $(a,b)=(1,1)$ and we can solve a Diophantine equation as follows. 
	
\begin{lemma}\label{8}
	For the Diophantine equation with two variables as follows
	\begin{align}x_{2}^4+x_{1}^2x_{2}^2+2x_{2}^2+x_{1}^2+2x_{1}+1=8x_{1}x_{2}^2,\label{dio B2} \end{align} all the positive integer solutions can be derived from the initial solution through a finite number of mutations of type $B_2$.
\end{lemma}
\begin{proof}	
Since $(1,1)$ is an initial solution, by \Cref{B2}, we obtain a sequence of solutions to \eqref{B2} through mutations as follows
	\begin{align}
		(1,1) \stackrel{\mu_{1}}{\longleftrightarrow} (2,1) \stackrel{\mu_{2}}{\longleftrightarrow} (2,3)\stackrel{\mu_{1}}{\longleftrightarrow} (5,3)\stackrel{\mu_{2}}{\longleftrightarrow} (5,2)\stackrel{\mu_{1}}{\longleftrightarrow} (1,2) \stackrel{\mu_{2}}{\longleftrightarrow} (1,1)\notag
		.\end{align}
	 Now, we claim that there are no other solutions. Assume that $(a,b)$ is another positive integer solution to \eqref{dio B2}. By \Cref{B2}, we deduce that $$(a',b)=\mu_{1}(a,b)=(\dfrac{b^2+1}{a},b)$$ is also a solution to \eqref{dio B2}. Notice that $$(b^2+1)a^2+(2-8b^2)a+(b^2+1)^2=0.$$
	 By Vieta's formulas, we get
	\begin{align}
		a+a^{\prime}=\dfrac{8b^2-2}{b^2+1}\geq 2\sqrt{aa^{\prime}}=2\sqrt{b^2+1},\notag
	\end{align}
	implying that $b\leq 3$. Similarly, we have 
	\begin{align}
		b^2+{b^{\prime}}^2=-a^2+8a-2\geq 2bb^{\prime}=2(a+1).\notag
	\end{align}
	Thus, we get $a\leq 5$, indicating that $a=4$ or $a=3$. However, it is direct that there are no positive integer solutions to $b$ for both cases. Consequently, we conclude that all the positive integer solutions to \eqref{dio B2} can be derived from the initial solution $(1,1)$ through a finite number of mutations.
\end{proof}
\subsection{Diophantine equations of type $G_2$}\

By \Cref{M action provides mutation invariants rank 2}, the Diophantine equations about $x_1$ and $x_2$ encoded with the cluster algebra of type $G_2$ are as follows
\begin{align}\Phi(F(c_{1;1}(x_1,x_2),c_{2;1}(x_1,x_2)),\cdots,F(c_{1;8}(x_1,x_2),c_{2;8}(x_1,x_2)))=\mcT(a,b), \label{eqG2}
\end{align}	
 for any symmetric polynomial $\Phi(X_1,\dots,X_{8})$, rational function $F(X_1,X_2)$ and $a,b\in \mbN$. Note that $(c_{1;i}(x_1,x_2),c_{2;i}(x_1,x_2))^8_{i=1}$ are 8 different clusters of type $G_2$. It is clear that $(a,b)$ is a solution to \eqref{eqG2} which is called the \emph{initial solution}. In particular, as \eqref{G2 invariant1} in \Cref{G2}, take $(a,b)=(1,1)$ and we can solve a Diophantine equation as follows. 
\begin{lemma}\label{11}
	For the Diophantine equation with two variables as follows
	\begin{align}x_{2}^4+x_{1}x_{2}^3+x_{2}^3+x_{1}^2x_{2}+2x_{1}x_{2}+x_{1}^2+x_{2}+2x_{1}+1=11x_{1}x_{2}^2, \label{dio G2} \end{align} all the positive integer solutions can be derived from the initial solution through a finite number of mutations of type $G_2$.
\end{lemma}
\begin{proof}
Since $(1,1)$ serves as an initial solution, by \Cref{G2}, a sequence of solutions to \eqref{dio G2} induced by mutations are given by
 \begin{align}
 	&(1,1) \stackrel{\mu_{1}}{\longleftrightarrow} (2,1) \stackrel{\mu_{2}}{\longleftrightarrow} (2,3)\stackrel{\mu_{1}}{\longleftrightarrow} (14,3)\stackrel{\mu_{2}}{\longleftrightarrow} (14,5)\notag \\& \stackrel{\mu_{1}}{\longleftrightarrow} (9,5) \stackrel{\mu_{2}}{\longleftrightarrow} (9,2)\stackrel{\mu_{1}}{\longleftrightarrow} (1,2)\stackrel{\mu_{2}}{\longleftrightarrow} (1,1) \notag
 	.\end{align}
 Now, we claim that there are no other solutions. Assume that $(a,b)$ is another positive integer solution to \eqref{dio G2}. By \Cref{G2}, we obtain that that $$(a',b)=\mu_{1}(a,b)=(\dfrac{b^3+1}{a},b)$$ is also a solution to \eqref{G2}. Notice that $$(b+1)a^2+(b^3-11b^2+2b+2)a+(b^4+b^3+b+1)=0.$$ By Vieta's formulas, we get 
 \begin{align}
 	a+a^{\prime}=\dfrac{-b^3+11b^2-2b-2}{b+1},\, aa^{\prime}=b^3+1. \notag
 \end{align} 
Since $a+a^{\prime}\geq2\sqrt{aa^{\prime}}$, we deduce that $b\leq 5$. Moreover, when $b=4$, there is no integer solution to  $5a^2-102a+325=0$. Consequently, we get all positive integer solutions to \eqref{dio G2} can be derived from the initial solution $(1,1)$ through a finite number of mutations.
\end{proof}
\begin{remark}
The Diophantine equations solved in \Cref{6}, \Cref{9}, \Cref{8} and \Cref{11} all possess finite positive integer solutions, which are derived from the initial solutions through finite mutations. Therefore, a natural question arises as follows.
\end{remark}
\begin{question}
	Is it true that all the solutions to the Diophantine equations \eqref{eqA1}, \eqref{eqA2}, \eqref{eqB2}, \eqref{eqG2} encoded with cluster algebras of finite type can be derived from the initial solution $(a,b)$ through finite mutations? 
\end{question}
\subsection{Diophantine equations of affine type}\

Firstly, we consider the affine $A_1^{(1)}$ case that $m=n=2$ in \eqref{matrix} and solve the Diophantine equation \eqref{dio 22} as follows. Note that the equation \eqref{dio 22} is also studied in \cite[Theorem 21]{GM} by generalized cluster algebras. In addition, it can be obtained by substituting $1$ for one of the variables in the Markov equation \eqref{markov}. 
\begin{lemma}\label{main2}
	For the Diophantine equation with two variables as follows
	\begin{align}x_1^2+x_2^2+1=3x_1x_2, \label{dio 22}\end{align} all the positive integer solutions can be derived from the initial solutions through a finite number of mutations of $\mcA$.
\end{lemma}
\begin{proof}
	 It is clear that $(1,1)$ is a solution and we call it an \emph{initial solution}. According to \Cref{AA} and \Cref{Integral}, we get a sequence of positive integer solutions to \eqref{dio 22}  by mutations as follows \begin{align} \cdots  \stackrel{\mu_{2}}{\longleftrightarrow} (5,2)\stackrel{\mu_{1}}{\longleftrightarrow} (1,2)\stackrel{\mu_{2}}{\longleftrightarrow} (1,1) \stackrel{\mu_{1}}{\longleftrightarrow} (2,1) \stackrel{\mu_{2}}{\longleftrightarrow} (2,5) \stackrel{\mu_{1}}{\longleftrightarrow}   \cdots .\label{22 seq}
		\end{align}
		Then, we claim that for any positive integer solution $(a,b)$, there is a sequence $(t_1,\dots,t_{r})\in (1,2)^{\times r}$ for some integer $r\geq 0$ such that $$(a,b)=(\mu_{t_{r}} \dots \mu_{t_{1}} )(1,1).$$ In fact, we can take induction on the maximum $m=\max(a,b)$. It is clear that when $m=1$, the claim holds. For $m>1$, without loss of generality, we can assume $a\geq b$. Let $(a^{\prime},b)=\mu_{1}(a,b)$. According to \Cref{Integral}, we get  $a^{\prime}\in \mathbb{N}_{+}$. 
		Considering the polynomial $g(\lambda) = \lambda^2 - 3b\lambda + b^2 + 1 $, we observe that \(a\) and \(a^{\prime}\) are two zeros.
		It is clear that $(a,b)=(2,1)=\mu_{1}(1,1)$ is the unique solution with $b=1$ except for the initial solution $(1,1)$. Now, assume that $b\neq 1$ and we get $$g(b)=b^2-3b^2+b^2+1=1-b^2<0.
		\notag $$ This infers that the polynomial $g(\lambda)$ has two distinct zero points and $b$ lies between them. Therefore, according to $a\geq b$, we get $a^{\prime}<b<a$. In particular, the maximum $m=a$ is strictly larger than $b$, which is $\max(a^{\prime},b)$. 
		By induction, there is a sequence \((t_1, \ldots, t_{r}) \in (1,2)^{\times r}\) for some $r\in \mbN_{+}$ such that \((a^{\prime},b) = (\mu_{t_{r}} \ldots \mu_{t_{1}} )(1,1)\). Therefore, we obtain that $$(a,b) = \mu_{1}(a^{\prime},b) = (\mu_{1} \mu_{t_{r}} \ldots \mu_{t_{1}} )(1,1).$$ It follows that all the positive integer solutions to \eqref{dio 22} are as \eqref{22 seq} and any one of them can be derived from the initial solution $(1,1)$ through a finite number of mutations.
\end{proof}	
\begin{remark}
Notably, mutations applied to the initial solution $(1,1)$ consistently maintain the integrality, generating an infinite sequence of positive integer solutions to \eqref{dio 22}.
\end{remark}
Now, consider the affine $A_2^{(2)}$ case that $m=1$, $n=4$ in \eqref{matrix}. We aim to solve the following Diophantine equation \begin{align}
	x_{2}^4+x_{1}^2+2x_{1}+1=5x_{1}x_{2}^2.\label{eq14}
\end{align}
	 Firstly, a preparatory lemma as follows is necessary. 
\begin{lemma}\label{ab}
	Let $(a,b)$ be a positive integer solution to \eqref{eq14}, where $a\neq 1$, $b\neq 1$.
	\begin{enumerate} 
	\item Take $(a^{\prime},b)=\mu_{1}(a,b)$, if $a>b^2$, then $a^{\prime}<b^2<a$; if $a<b^2$, then $a^{\prime}>b^2>a$.
	\item Take $(a,b^{\prime})=\mu_{2}(a,b)$, if $a>b^2$, then ${b^{\prime}}^2>a>b^2$; if $a<b^2$, then ${b^{\prime}}^2<a<b^2$.
	\end{enumerate} 
\end{lemma} 
\begin{proof}
	We first consider the case that $(a^{\prime},b)=\mu_{1}(a,b)$. Notice that $aa^{\prime}=b^4+1$. If $a>b^2$, we claim that $a^{\prime}<b^2$. Otherwise, we have $a^{\prime}\geq b^2>1$, which implies that $$aa^{\prime}\geq (b^2+1)b^2=b^4+b^2>b^4+1,$$ which leads to a contradiction. If $a<b^2$, it follows that $$a^{\prime}=\dfrac{b^4+1}{a}>\dfrac{b^4+1}{b^2}>b^2.$$
	Now, consider the second case that $(a,b^{\prime})=\mu_{2}(a,b)$, which implies that $bb^{\prime}=a+1$. Consequently, we have ${b^{\prime}}^2a>(a+1)^2$, which implies that $${b^{\prime}}^2>\dfrac{(a+1)^2}{a}>a>b^2.$$ If $a<b^2$, it is clear that $$b^{\prime}=\dfrac{a+1}{b}<\dfrac{b^2+1}{b}=b+\dfrac{1}{b},$$ which implies that  $b^{\prime} \leq b$. Now we claim that $b^{\prime}< b$. Otherwise, since $b^2$ and ${b^{\prime}}^2$ are two zero points of the quadratic polynomial $g(\lambda)=\lambda^2-5a\lambda+(a+1)^2$, we obtain that $$25a^2-4(a+1)^2=21a^2-8a-4=0.$$ However, $a$ cannot be a positive integer, which leads to a contradiction. Moreover, $$(a+1)^2=b^2{b^{\prime}}^2>{b^{\prime}}^4,$$  which implies that ${b^{\prime}}^2 \leq a$. If ${b^{\prime}}^2=a$, by Vieta's formulas, we get $b^2=5a-a=4a$. It follows that $b^2{b^{\prime}}^2=4a^2=(a+1)^2$, which implies that $a=1$. Hence, it contradicts with $a\neq 1$ and we get ${b^{\prime}}^2 <a$.
\end{proof}

\begin{lemma}\label{main3}
	All the positive integer solutions to \eqref{eq14} can be derived from the initial solutions through a finite number of mutations of $\mcA$.
\end{lemma}
\begin{proof}
	It is clear that $(1,1)$ is a solution, which we refer to as an \emph{initial solution}. According to \Cref{14} and \Cref{int2}, we obtain a sequence of positive integer solutions to \eqref{eq14} by mutations as follows
	\begin{align} \cdots  \stackrel{\mu_{2}}{\longleftrightarrow}(3,2)\stackrel{\mu_{1}}{\longleftrightarrow} (1,2)\stackrel{\mu_{2}}{\longleftrightarrow} (1,1) \stackrel{\mu_{1}}{\longleftrightarrow} (2,1) \stackrel{\mu_{2}}{\longleftrightarrow} (2,3) \stackrel{\mu_{1}}{\longleftrightarrow} (41,3)  \stackrel{\mu_{2}}{\longleftrightarrow}  \cdots. \label{solution14}
\end{align}
	Now, we claim that for any positive integer solution $(a,b)$, there exists a sequence $(t_1,\dots,t_{r})\in (1,2)^{\times r}$ for some integer $r\geq 0$ such that $(a,b)=(\mu_{t_{r}} \dots \mu_{t_{1}} )(1,1)$. We take the induction on the maximum $m=\max(a,b^2)$. In the following, we assume that $a\neq 1$ and $b\neq 1$ since the solutions of these cases are clear. It is direct that $m>1$ and there are three cases to consider: $a>b^2$, $a<b^2$ or $a=b^2$. 
	
	If $a>b^2$, take $(a^{\prime},b)=\mu_{1}(a,b)$. By \Cref{ab} we have $m=a$ is strictly larger than $b^2$, which is $\max(a^{\prime},b^2)$. By induction, there is a sequence $(t_1,\dots,t_{r})\in (1,2)^{\times r}$ for some $r\in \mathbb{N}_{+}$ such that $(a^{\prime},b)=(\mu_{t_{r}} \dots  \mu_{t_{1}} )(1,1)$. Hence $(a,b)=\mu_{1}(a^{\prime},b)=(\mu_{1} \mu_{t_{r}} \dots  \mu_{t_{1}} )(1,1)$. 
	
	If $a<b^2$, take $(a,b^{\prime})=\mu_{2}(a,b)$. By \Cref{ab} we have $m=b^2$ is strictly larger than $a$, which is $\max(a,{b^{\prime}}^2)$. By induction, there is a sequence $(t_1,\dots,t_{r})\in (1,2)^{\times r}$ for some $r\in \mathbb{N}_{+}$ such that $(a,b^{\prime})=(\mu_{t_{r}} \dots  \mu_{t_{1}} )(1,1)$. Hence $(a,b)=\mu_{2}(a,b^{\prime})=(\mu_{2} \mu_{t_{r}} \dots  \mu_{t_{1}} )(1,1)$. 
	
	If $a=b^2$, it follows that $a=b=1$, which contradicts with the assumption. 
	
	Therefore, all the positive integer solutions to \eqref{eq14} are as \eqref{solution14} and any one of them can be derived from the initial solution $(1,1)$ through a finite number of mutations.
\end{proof}
\begin{remark}
	Note that the Diophantine equation \eqref{eq14} is also studied in \cite[Theorem 21]{GM} by generalized cluster algebras. In addition, it can be obtained by substituting $1$ for either $Y$ or $Z$ in the Lampe equation \eqref{lampe-equation}. 

\end{remark}
\subsection*{Acknowledgements}
The authors would like to express heartleft thanks to  Yu Ye and Zhe Sun from University of Science and Technology of China for their support and help. The authors also want to thank referees for their suggestions about the citations and grammars of this paper. This work is partially supported by the National Natural Science Foundation of China (Grant Nos. 12131015 and 12371042). 

\end{document}